\pdfoutput=1
\documentclass[11pt]{article}
\usepackage[a4paper,margin=1in]{geometry}
\usepackage{amsmath,amssymb,amsthm,mathtools}
\usepackage{enumitem}
\usepackage{hyperref}
\usepackage{microtype}
\usepackage{tikz}
\usetikzlibrary{arrows.meta,calc,decorations.pathreplacing,positioning,shapes.geometric}

\hypersetup{
  colorlinks=true,
  linkcolor=blue,
  citecolor=blue,
  urlcolor=blue,
  pdftitle={The Antipodal Defect of a Convex Polyhedron},
  pdfauthor={Kieu Gia Thinh Phat},
  pdfsubject={Convex polyhedra, antipodal pairs, normal fans, square complexes, and integral homology},
  pdfkeywords={convex polyhedra, antipodal pairs, exact normal strata, normal fans, square complexes, cellular homology, Smith normal form}
}

\newtheorem{theorem}{Theorem}[section]
\newtheorem{corollary}[theorem]{Corollary}
\newtheorem{lemma}[theorem]{Lemma}
\newtheorem{proposition}[theorem]{Proposition}

\newtheorem{problem}[theorem]{Problem}
\theoremstyle{definition}
\newtheorem{definition}[theorem]{Definition}
\newtheorem{example}[theorem]{Example}
\theoremstyle{remark}
\newtheorem{remark}[theorem]{Remark}

\DeclareMathOperator{\Vertices}{Vert}
\DeclareMathOperator{\rank}{rank}

\DeclareMathOperator{\coker}{coker}

\newcommand{\SqBd}{\mathsf{D}_P}
\newcommand{\Xdir}{\widetilde X}

\title{The Antipodal Defect of a Convex Polyhedron}
\author{Kieu Gia Thinh Phat\thanks{Email: \texttt{phat.kgt@gmail.com}}}
\date{July 2026}

\begin{document}
\maketitle

\begin{abstract}
Problem C7 from the 2006 IMO Shortlist gives
\(A(P)-B(P)=V(P)-1\) for a generic convex polyhedron, where
\(A(P)\) counts antipodal vertex pairs and \(B(P)\) counts antipodal
edge-midpoint pairs.  For an arbitrary convex polyhedron
\(P\subset\mathbb R^3\), define
\[
        \delta(P)=V(P)-1-A(P)+B(P).
\]
We construct a two-dimensional antipodal square complex \(X(P)\) and prove
\[
H_0(X(P);\mathbb Z)\cong\mathbb Z,\qquad
H_1(X(P);\mathbb Z)\cong\mathbb Z/2,\qquad
H_2(X(P);\mathbb Z)\cong\mathbb Z^{\delta(P)}.
\]
Consequently \(\delta(P)\ge0\), extending the generic identity to the
inequality \(A(P)-B(P)\le V(P)-1\).  The proof uses a directed double
cover, a polyhedral support blow-up over the normal sphere, and the
Vietoris--Begle mapping theorem.  Independently, Euler integration on the
projective normal fan gives an exact local formula for the defect; in
three dimensions only exact edge--facet and facet--facet opposite pairs
contribute.  We also determine the integral image of the square-boundary
map: it is the even-cycle lattice of the antipodal graph, with nonzero
Smith invariant factors \(1,\ldots,1,2\).  Applications include a
zero-defect criterion, centrally symmetric and extremal formulas, and an
explicit description of defects and primitive belts for pyramids over
polygons.
\end{abstract}

\medskip
\noindent\textbf{2020 Mathematics Subject Classification.}
Primary 52B10; Secondary 52B05, 55N10, 05B35.

\smallskip
\noindent\textbf{Keywords.}
convex polyhedra; antipodal pairs; exact normal strata; normal fans; square complexes; cellular homology; Smith normal form.

\section{Introduction}

A pair of points of a convex polyhedron \(P\subset\mathbb R^3\) is
called \emph{antipodal} if two parallel supporting planes of \(P\), one
through each point, contain \(P\) between them.  Throughout the
three-dimensional part of the paper, ``polyhedron'' means a compact
full-dimensional convex polytope.  In arbitrary dimension we use the
term convex \(d\)-polytope.

Problem C7 from the 2006 IMO Shortlist assumes that no two edges of
\(P\) are parallel and that no edge is parallel to a non-adjacent
facet.  If \(A\) is the number of unordered antipodal vertex pairs,
\(B\) is the number of unordered antipodal edge-midpoint pairs, and
\(V\) is the number of vertices, then the problem asks one to prove
\[
        A-B=V-1.
\]
The standard normal-fan proof is an Euler-characteristic count on the
sphere.  The purpose of this paper is to understand what replaces this
identity when the genericity assumptions are removed.

There is a substantial literature on antipodal and strictly antipodal
pairs in finite point sets and convex position; see, for example,
\cite{MartiniSoltan2005,MakaiMartiniNguyenSoltanTalata2021}.  Our
comparison is different: a single polyhedron is fixed, and vertex
pairs are compared with edge-midpoint pairs.  The discrepancy is
encoded by the following invariant.

\begin{definition}[Antipodal defect]
For a convex polyhedron \(P\subset\mathbb R^3\), define
\[
        \delta(P)=V(P)-1-A(P)+B(P).
\]
We call \(\delta(P)\) the \emph{antipodal defect} of \(P\).
\end{definition}

All pairs counted in \(A(P)\) and \(B(P)\) consist of distinct
objects.  The IMO identity says that \(\delta(P)=0\) in C7-generic
position, whereas symmetric examples such as the cube have positive
defect.

The central construction is a two-dimensional CW complex \(X(P)\).
Its vertices are the vertices of \(P\), its edges are the antipodal
vertex pairs, and its square cells are indexed by antipodal pairs of
edge midpoints.  The attaching map of a square is the four-cycle
formed by the endpoints of the corresponding two edges.  The main
theorem identifies the defect with the top integral homology of this
complex.

\begin{theorem}[Integral antipodal homology theorem]\label{thm:main-strong}
For every convex polyhedron \(P\subset\mathbb R^3\), the antipodal
square complex \(X(P)\) satisfies
\[
H_i(X(P);\mathbb Z)\cong
\begin{cases}
\mathbb Z, & i=0,\\
\mathbb Z/2, & i=1,\\
\mathbb Z^{\delta(P)}, & i=2,\\
0, & i\ge 3.
\end{cases}
\]
\end{theorem}

\begin{corollary}[Antipodal defect inequality]\label{cor:intro-defect}
For every convex polyhedron \(P\subset\mathbb R^3\),
\[
        \delta(P)=V(P)-1-A(P)+B(P)\ge0.
\]
Equivalently,
\[
        A(P)-B(P)\le V(P)-1.
\]
More precisely,
\[
        \delta(P)=\beta_2(X(P);\mathbb Q).
\]
\end{corollary}

\begin{corollary}[Zero defect and projective-plane homology]\label{cor:zero-defect-rp2}
A convex polyhedron \(P\subset\mathbb R^3\) satisfies \(\delta(P)=0\)
if and only if \(X(P)\) has the integral homology of
\(\mathbb{RP}^2\).
\end{corollary}

The proof uses the directed square complex \(\widetilde X(P)\), a
connected two-sheeted cover of \(X(P)\).  We attach to every cell of
\(\widetilde X(P)\) the spherical set of normal directions that witness
that cell.  The resulting support blow-up has two projections: one has
contractible fibers, while the other has connected fibers with
vanishing first homology.  A degree-one Vietoris--Begle theorem then
gives \(H_1(\widetilde X(P);k)=0\) over every field \(k\).  The covering
and cellular chain complexes yield the integral statement.

The paper develops two further aspects of the same geometry.  First,
Euler integration on the projective normal fan gives the exact local formula
\[
\delta(P)=
\sum_{\{F,G\}\in\mathcal O_{\rm ex}(P)}
\bigl(e(F)e(G)-(v(F)-1)(v(G)-1)\bigr).
\]
In dimension three, the nonzero terms are precisely the exact
edge--facet and facet--facet pairs.  This localizes the defect and yields
zero-defect criteria, extremal bounds, and closed formulas for symmetric
examples.  Second, the integral image of the square-boundary map is exactly
the even-cycle lattice in the antipodal graph; hence its nonzero Smith
factors are \(1,\ldots,1,2\).  The resulting boundary matroid leads to a
notion of primitive belt, which can be described completely for pyramids
over polygons.

The argument is organized as follows.  Section~2 fixes the normal-fan
notation.  Section~3 proves the exact local defect formula, and Section~4
recovers the generic Euler relation in arbitrary dimension.  Section~5
defines the antipodal square complex, while Sections~6 and~7 give examples
and the difference-body viewpoint.  Section~8 constructs the support
blow-up and proves the integral homology theorem.  Section~9 determines the
integral square-boundary lattice and its Smith normal form.  Section~10
develops the boundary matroid, primitive belts, and the polygonal pyramid
family.  The final section summarizes the main consequences and records
several focused open problems.

\section{Normal fans and exact antipodal face pairs}

Let $P\subset\mathbb R^d$ be a full-dimensional convex polytope.  We use standard terminology from the theory of convex polytopes and normal fans; see, for example, \cite{ziegler,grunbaum}.  For $u\in S^{d-1}$, define the support function and support face by
\[
        h_P(u)=\max_{x\in P}\langle u,x\rangle,
        \qquad
        F_P(u)=\{x\in P: \langle u,x\rangle=h_P(u)\}.
\]
The face $F_P(u)$ is also called the exposed face in direction $u$.

\begin{definition}[Closed and exact spherical normal cells]
For a nonempty proper face $F$ of $P$, its closed outward spherical normal cell is
\[
        N^+(F)=\{u\in S^{d-1}:F\subseteq F_P(u)\}.
\]
Its exact, or relatively open, spherical normal cell is
\[
        N^\circ(F)=\{u\in S^{d-1}:F_P(u)=F\}.
\]
We also write
\[
        N^-(F)=-N^+(F),
        \qquad
        -N^\circ(F)=\{-u:u\in N^\circ(F)\}.
\]
\end{definition}

The exact cells $N^\circ(F)$ are pairwise disjoint and form the relatively open cells of the spherical normal fan.  If $\dim F=i$, then
\[
        \dim N^\circ(F)=d-1-i.
\]
The use of exact cells is essential: closed normal cells overlap along their boundaries and therefore should not be used for cell-counting without additional conventions.

\begin{definition}[Exact ordered antipodal face pair]
Two proper faces $F,G$ of $P$ form an \emph{exact ordered antipodal pair} if
\[
        N^\circ(F)\cap\bigl(-N^\circ(G)\bigr)\ne\varnothing.
\]
Equivalently, there exists a direction $u\in S^{d-1}$ such that
\[
        F_P(u)=F,
        \qquad
        F_P(-u)=G.
\]
\end{definition}

For ordinary point-antipodality, closed normal cells are enough: two faces $F,G$ can be supported by opposite parallel planes if and only if $N^+(F)\cap(-N^+(G))$ is nonempty.  For Euler cell-counting, however, exact cells are the correct objects because each direction $u$ has a unique pair of support faces $(F_P(u),F_P(-u))$.

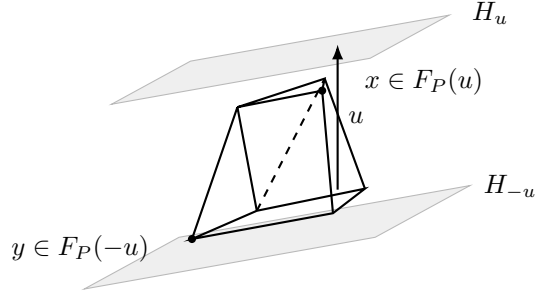
\begin{figure}[ht]
\centering
\begin{tikzpicture}[x={(0.9cm,0.16cm)},y={(0.55cm,0.30cm)},z={(0cm,0.88cm)},scale=0.95,every node/.style={font=\small}]
  % support planes
  \coordinate (L1) at (-2.3,-1.2,-0.1);
  \coordinate (L2) at (2.2,-1.2,-0.1);
  \coordinate (L3) at (2.2,1.2,-0.1);
  \coordinate (L4) at (-2.3,1.2,-0.1);
  \coordinate (U1) at (-2.0,-1.0,2.7);
  \coordinate (U2) at (2.0,-1.0,2.7);
  \coordinate (U3) at (2.0,1.0,2.7);
  \coordinate (U4) at (-2.0,1.0,2.7);
  \fill[gray!15,opacity=.8] (L1)--(L2)--(L3)--(L4)--cycle;
  \fill[gray!15,opacity=.8] (U1)--(U2)--(U3)--(U4)--cycle;
  \draw[gray!55] (L1)--(L2)--(L3)--(L4)--cycle;
  \draw[gray!55] (U1)--(U2)--(U3)--(U4)--cycle;

  % polyhedron
  \coordinate (A) at (-1.0,-0.6,0.25);
  \coordinate (B) at (1.15,-0.55,0.25);
  \coordinate (C) at (0.85,0.75,0.25);
  \coordinate (D) at (-0.7,0.55,0.25);
  \coordinate (E) at (-0.45,-0.35,2.15);
  \coordinate (F) at (0.8,-0.25,2.15);
  \coordinate (G) at (0.35,0.55,2.15);
  \draw[thick] (A)--(B)--(C)--(D)--cycle;
  \draw[thick] (E)--(F)--(G)--cycle;
  \draw[thick] (A)--(E) (B)--(F) (C)--(G) (D)--(E);
  \draw[thick,dashed] (D)--(G);

  \coordinate (x) at (F);
  \coordinate (y) at (A);
  \fill (x) circle (1.6pt);
  \fill (y) circle (1.6pt);
  \node[anchor=west] at ($(x)+(0.35,0.25,0)$) {$x\in F_P(u)$};
  \node[anchor=east] at ($(y)+(-0.35,-0.25,0)$) {$y\in F_P(-u)$};

  \draw[-{Latex[length=2mm]},thick] (0,1.45,0.15)--(0,1.45,2.45) node[midway,right] {$u$};
  \node[anchor=west] at (2.2,1.0,2.7) {$H_u$};
  \node[anchor=west] at (2.35,1.05,-0.1) {$H_{-u}$};
\end{tikzpicture}
\caption{Antipodal points are witnessed by two opposite parallel supporting planes.  Equivalently, the directed difference \(x-y\) lies on the boundary of the difference body \(P-P\).}
\label{fig:support-planes}
\end{figure}

\section{The exact local defect formula}\label{sec:exact-local-defect}

This section gives a second, purely local description of the antipodal defect.  It is independent of the square-complex homology theorem: the proof is an Euler-integration argument over the projective normal fan.

Let \(P\subset\mathbb R^3\) be a full-dimensional convex polyhedron and put
\[
        M=\mathbb{RP}^2.
\]
For a projective direction \([u]\in M\), define the unordered opposite support pair
\[
        \{F_u,G_u\}=\{F_P(u),F_P(-u)\}.
\]
The exact normal cells of \(P\), together with their antipodal images, induce a finite cell decomposition of \(M\).  On each relatively open cell the unordered pair \(\{F_u,G_u\}\) is constant.

\begin{definition}[Exact opposite face pairs]
Let \(\mathcal O_{\rm ex}(P)\) denote the set of unordered pairs \(\{F,G\}\) of proper faces of \(P\) such that
\[
        N^\circ(F)\cap -N^\circ(G)\ne\varnothing.
\]
Equivalently, \(\{F,G\}\in\mathcal O_{\rm ex}(P)\) if and only if \(\{F,G\}=\{F_P(u),F_P(-u)\}\) for some \(u\in S^2\).
\end{definition}

For a face \(F\), write
\[
        v(F)=|\operatorname{Vert}(F)|,
        \qquad
        e(F)=|\operatorname{Edge}(F)|.
\]
Thus a vertex has \(v(F)=1,e(F)=0\), while an edge has \(v(F)=2,e(F)=1\).  For an exact opposite pair define the local weight
\[
        \rho(F,G)=e(F)e(G)-(v(F)-1)(v(G)-1).
\]

\begin{theorem}[Exact local defect formula]\label{thm:exact-local-defect}
For every full-dimensional convex polyhedron \(P\subset\mathbb R^3\),
\[
        \delta(P)=
        \sum_{\{F,G\}\in\mathcal O_{\rm ex}(P)}
        \bigl(e(F)e(G)-(v(F)-1)(v(G)-1)\bigr).
\]
Equivalently,
\[
        \delta(P)=
        N_{EF}^{\rm ex}(P)
        +
        \sum_{\{F,G\}\in FF^{\rm ex}(P)}
        \bigl(e(F)+e(G)-1\bigr),
\]
where \(N_{EF}^{\rm ex}(P)\) is the number of exact edge--facet opposite pairs and \(FF^{\rm ex}(P)\) is the set of exact facet--facet opposite pairs.
\end{theorem}

\begin{proof}
We use Euler integration with compactly supported Euler characteristic over the exact cell decomposition of \(M=\mathbb{RP}^2\); see, for example, \cite{BaryshnikovGhrist2009}.  If \(\varphi\) is constant on each open cell \(\sigma\), then
\[
        \int_M \varphi\,d\chi=
        \sum_\sigma (-1)^{\dim\sigma}\varphi(\sigma).
\]
In particular,
\[
        \int_M 1\,d\chi=\chi(\mathbb{RP}^2)=1.
\]

On a stratum with exact support pair \(\{F,G\}\), define constructible functions
\[
        v=v(F)+v(G),
        \qquad
        a=v(F)v(G),
        \qquad
        b=e(F)e(G).
\]
We claim that
\[
        \int_M v\,d\chi=V(P),
        \qquad
        \int_M a\,d\chi=A(P),
        \qquad
        \int_M b\,d\chi=B(P).
\]

For the first identity, write
\[
        v([u])=
        \#\{x\in V(P):x\in F_P(u)\text{ or }x\in F_P(-u)\}.
\]
Thus \(v=\sum_{x\in V(P)}\mathbf 1_{W_x}\), where
\[
        W_x=
        \{[u]\in\mathbb{RP}^2:x\in F_P(u)\text{ or }x\in F_P(-u)\}.
\]
The set \(W_x\) is the projectivization of \(N^+(x)\cup -N^+(x)\).  Since \(P\) is full-dimensional, \(N^+(x)\) and \(-N^+(x)\) are disjoint on \(S^2\), and the quotient is homeomorphic to the closed spherical normal cell \(N^+(x)\).  This cell is a nonempty convex spherical polygon, hence has Euler characteristic \(1\).  Therefore
\[
        \int_M v\,d\chi=
        \sum_{x\in V(P)}\chi(W_x)=V(P).
\]

For vertex pairs, let \(\{x,y\}\) be an unordered pair of distinct vertices and define its witness set
\[
        W_{x,y}=\{[u]\in\mathbb{RP}^2:
        x\in F_P(u),\ y\in F_P(-u)
        \text{ or }
        y\in F_P(u),\ x\in F_P(-u)\}.
\]
Then \(a=\sum_{\{x,y\}}\mathbf 1_{W_{x,y}}\).  If \(\{x,y\}\) is not antipodal, then \(W_{x,y}\) is empty.  If it is antipodal, the ordered witness set
\[
        \widetilde W_{x,y}=N^+(x)\cap -N^+(y)\subset S^2
\]
is a nonempty intersection of closed convex spherical normal cells, hence is a contractible spherical polyhedral cell.  Passing to \(\mathbb{RP}^2\) identifies \(\widetilde W_{x,y}\) with its antipodal copy and gives \(W_{x,y}\cong \widetilde W_{x,y}\).  Thus \(\chi(W_{x,y})=1\) precisely for antipodal vertex pairs.  Hence
\[
        \int_M a\,d\chi=A(P).
\]

The argument for edge-midpoint pairs is identical.  If \(\{e,f\}\) is an unordered pair of distinct edges, let
\[
        W_{e,f}=\{[u]\in\mathbb{RP}^2:
        e\subset F_P(u),\ f\subset F_P(-u)
        \text{ or }
        f\subset F_P(u),\ e\subset F_P(-u)\}.
\]
Then \(b=\sum_{\{e,f\}}\mathbf 1_{W_{e,f}}\).  The set \(W_{e,f}\) is empty unless the two edge midpoints are antipodal.  If they are antipodal, the ordered witness set
\[
        \widetilde W_{e,f}=N^+(e)\cap -N^+(f)
\]
is a nonempty closed convex spherical cell, hence has Euler characteristic \(1\), and its unordered projective quotient has the same Euler characteristic.  Therefore
\[
        \int_M b\,d\chi=B(P).
\]

Using the definition of \(\delta(P)\), we obtain
\[
\begin{aligned}
        \delta(P)
        &=V(P)-1-A(P)+B(P)\\
        &=\int_M (v-1-a+b)\,d\chi.
\end{aligned}
\]
On a stratum with exact pair \(\{F,G\}\),
\[
\begin{aligned}
        v-1-a+b
        &=v(F)+v(G)-1-v(F)v(G)+e(F)e(G)\\
        &=e(F)e(G)-(v(F)-1)(v(G)-1)\\
        &=\rho(F,G).
\end{aligned}
\]
Thus
\[
        \delta(P)=
        \sum_\sigma (-1)^{\dim\sigma}\rho(F_\sigma,G_\sigma).
\]
In dimension three, \(\rho(F,G)\ne0\) only for edge--facet and facet--facet exact pairs.  These exact strata are zero-dimensional in \(\mathbb{RP}^2\), since \(\dim N^\circ(F)=2-\dim F\).  Hence every nonzero contribution has positive sign.  This proves the first formula.

The second formula follows by evaluating \(\rho\) on the possible face types.  For an edge and an \(r\)-gonal facet,
\[
        \rho=1\cdot r-(2-1)(r-1)=1.
\]
For an \(m\)-gonal facet and an \(r\)-gonal facet,
\[
        \rho=mr-(m-1)(r-1)=m+r-1.
\]
All remaining face-type pairs have \(\rho=0\).
\end{proof}

\begin{corollary}[Facet-opposite formula]\label{cor:facet-opposite-formula}
Let \(\mathcal F(P)\) be the set of facets of \(P\).  For a facet \(F\), let \(n_F\) be its outward unit normal and define its opposite support face by
\[
        F^*=F_P(-n_F).
\]
Let \(\mathcal F_{\operatorname{opp}}(P)\) be the set of unordered pairs \(\{F,G\}\) of opposite facets, equivalently \(F^*=G\) and \(G^*=F\).  Then
\[
        \delta(P)=
        \#\{F\in\mathcal F(P):F^*\text{ is an edge}\}
        +
        \sum_{\{F,G\}\in \mathcal F_{\rm opp}(P)}
        \bigl(e(F)+e(G)-1\bigr).
\]
\end{corollary}

\begin{proof}
By Theorem~\ref{thm:exact-local-defect}, only exact edge--facet and exact facet--facet pairs contribute.  An exact edge--facet pair has a unique facet member \(F\), and it occurs exactly when the support face opposite the outward normal of \(F\) is an edge, namely when \(F^*\) is an edge.  This gives the first term.  Exact facet--facet pairs are precisely unordered opposite facet pairs \(\{F,G\}\), and each such pair contributes \(e(F)+e(G)-1\).
\end{proof}

\begin{corollary}[Zero-defect criterion]\label{cor:facet-opposite-zero}
For a convex polyhedron \(P\),
\[
        \delta(P)=0
        \quad\Longleftrightarrow\quad
        F_P(-n_F)\text{ is a vertex for every facet }F\text{ of }P.
\]
Equivalently, zero defect means that no facet of \(P\) has an opposite support face which is an edge or a facet.
\end{corollary}

\begin{proof}
The facet-opposite formula is a sum of nonnegative terms.  Hence \(\delta(P)=0\) exactly when there are no facets whose opposite support face is an edge and no opposite facet pairs.  Since the opposite support face of a facet in dimension three is a vertex, an edge, or a facet, this is equivalent to saying that every facet has a vertex as its opposite support face.
\end{proof}

\begin{corollary}[Centrally symmetric polyhedra]\label{cor:centrally-symmetric-defect}
Let \(P\subset\mathbb R^3\) be centrally symmetric, and let \(E(P)\) and \(f_2(P)\) denote the numbers of edges and facets of \(P\), respectively.  Then
\[
        \delta(P)=
        \sum_{\{F,-F\}}(2e(F)-1)
        =2E(P)-\frac{f_2(P)}2,
\]
where the sum is over unordered pairs of opposite facets.
\end{corollary}

\begin{proof}
After translating, assume that \(P=-P\).  If \(F\) is a facet, then the opposite support face \(F^*\) is the facet \(-F\).  Thus the facet-opposite formula gives
\[
        \delta(P)=\sum_{\{F,-F\}}(e(F)+e(-F)-1)
        =\sum_{\{F,-F\}}(2e(F)-1).
\]
Since opposite facets have the same number of edges and every facet belongs to exactly one opposite pair,
\[
        \sum_{\{F,-F\}}2e(F)=\sum_{F\text{ facet}}e(F)=2E(P),
\]
because every edge lies in exactly two facets.  The number of opposite facet pairs is \(f_2(P)/2\).  Therefore \(\delta(P)=2E(P)-f_2(P)/2\).
\end{proof}

\begin{remark}[Relation to the homological theorem]
Theorem~\ref{thm:exact-local-defect} gives a local normal-fan formula for the same integer computed homologically by Theorem~\ref{thm:main-strong}.  The identity
\[
        \delta(P)=\beta_2(X(P);\mathbb Q)
\]
explains why the defect is a Betti number, while the exact local formula explains where the defect is supported in the projective normal fan.
\end{remark}

\subsection{Facet-normal profiles and the spherical normal graph}

The facet-opposite formula can be sharpened into a normal-fan profile.  For a facet $F$ of $P$, let
\[
        F^*=F_P(-n_F)
\]
be the support face opposite to the outward unit normal of $F$.  Define
\[
\mathcal F_i(P)=\{F\text{ facet of }P:\dim F^*=i\},\qquad i=0,1,2.
\]
Thus $\mathcal F_0$ consists of facets opposite to vertices, $\mathcal F_1$ consists of facets opposite to edges, and $\mathcal F_2$ consists of facets opposite to facets.

\begin{theorem}[Facet-normal profile]
For every convex polyhedron $P\subset\mathbb R^3$,
\[
        \delta(P)=\sum_{F\text{ facet}}\omega(F),
\]
where
\[
\omega(F)=
\begin{cases}
0, & F^*\text{ is a vertex},\\[1mm]
1, & F^*\text{ is an edge},\\[1mm]
e(F)-\dfrac12, & F^*\text{ is a facet}.
\end{cases}
\]
Equivalently, if $f_i^*=|\mathcal F_i(P)|$ and $I_2=\sum_{F\in\mathcal F_2(P)}e(F)$, then
\[
        \delta(P)=f_1^*+I_2-\frac{f_2^*}{2}.
\]
\end{theorem}

\begin{proof}
The facets in $\mathcal F_1(P)$ are exactly those which contribute the edge--facet term in the facet-opposite formula, and each contributes $1$.  The facets in $\mathcal F_2(P)$ occur in unordered opposite pairs $\{F,G\}$, because $F^*=G$ implies $G^*=F$.  The pair $\{F,G\}$ contributes $e(F)+e(G)-1$, which equals
\[
        \left(e(F)-\frac12\right)+\left(e(G)-\frac12\right).
\]
Facets in $\mathcal F_0(P)$ do not contribute.  This proves both forms of the formula.
\end{proof}

\begin{definition}[Facet-paired polyhedron]
A convex polyhedron $P$ is called \emph{facet-paired} if $F^*$ is a facet for every facet $F$ of $P$.
\end{definition}

\begin{corollary}[Extremal bound]
For every convex polyhedron $P\subset\mathbb R^3$,
\[
        0\le \delta(P)\le M(P):=2E(P)-\frac{f_2(P)}2.
\]
Moreover,
\[
        \delta(P)=0
        \quad\Longleftrightarrow\quad
        F^*\text{ is a vertex for every facet }F,
\]
and
\[
        \delta(P)=M(P)
        \quad\Longleftrightarrow\quad
        P\text{ is facet-paired}.
\]
\end{corollary}

\begin{proof}
The lower bound is the nonnegativity theorem, or directly the facet-normal profile.  For the upper bound, compare each facet weight with $e(F)-1/2$.  Since $e(F)\ge3$ for every facet,
\[
        \omega(F)\le e(F)-\frac12.
\]
Summing over all facets gives
\[
        \delta(P)\le \sum_{F\text{ facet}}\left(e(F)-\frac12\right)=2E(P)-\frac{f_2(P)}2,
\]
because every edge is contained in two facets.  Equality holds exactly when every facet has the third type of contribution, namely when every facet is opposite to a facet.  The zero criterion is Corollary~\ref{cor:facet-opposite-zero}, equivalently the zero case of the facet-normal profile.
\end{proof}

\begin{corollary}[Facet-paired simple and simplicial cases]
If $P$ is facet-paired, then
\[
        \delta(P)=2E(P)-\frac{f_2(P)}2.
\]
If, in addition, $P$ is simple, then
\[
        \delta(P)=\frac{11V(P)}4-1=\frac{11f_2(P)}2-12.
\]
If, in addition, $P$ is simplicial, then
\[
        \delta(P)=\frac{5f_2(P)}2=5V(P)-10.
\]
\end{corollary}

\begin{proof}
The first formula is the equality case of the extremal bound.  If $P$ is simple, then $2E=3V$ and Euler's formula gives $f_2=2+V/2$, hence
\[
        \delta=3V-\frac12\left(2+\frac V2\right)=\frac{11V}{4}-1.
\]
The expression in terms of $f_2$ follows from $V=2f_2-4$.  If $P$ is simplicial, then $2E=3f_2$, so
\[
        \delta=3f_2-\frac{f_2}{2}=\frac{5f_2}{2}.
\]
Since a simplicial three-polyhedron satisfies $f_2=2V-4$, this is $5V-10$.
\end{proof}

Let $\mathcal N(P)$ denote the spherical normal fan of $P$, and let
\[
        G_P=\mathcal N(P)^{(1)}
\]
be its one-skeleton on $S^2$.  A vertex $x$ of $G_P$ is the outward unit normal of a unique facet $F_x$, and
\[
        \deg_{G_P}(x)=e(F_x).
\]
The antipodal point $-x$ lies in exactly one cell of $\mathcal N(P)$.

\begin{theorem}[Spherical normal-graph profile]
For every convex polyhedron $P\subset\mathbb R^3$,
\[
\boxed{
        \delta(P)=
        \#\{x\in V(G_P):-x\in E(G_P)^\circ\}
        +
        \sum_{\{x,-x\}\subset V(G_P)}
        \bigl(\deg x+\deg(-x)-1\bigr).
}
\]
Equivalently,
\[
        \delta(P)=\sum_{x\in V(G_P)}\Omega(x),
\]
where
\[
\Omega(x)=
\begin{cases}
0, & -x\text{ lies in a two-cell of }\mathcal N(P),\\[1mm]
1, & -x\text{ lies in the relative interior of an edge of }G_P,\\[1mm]
\deg(x)-\dfrac12, & -x\in V(G_P).
\end{cases}
\]
\end{theorem}

\begin{proof}
The normal vertex $x$ corresponds to a facet $F_x$.  The cell of $\mathcal N(P)$ containing $-x$ is dual to the support face $F_x^*=F_P(-x)$.  Thus $-x$ lies in a two-cell, an open edge, or a vertex of the normal fan according as $F_x^*$ is a vertex, an edge, or a facet.  The one-vertex weight $\Omega(x)$ is therefore exactly the facet-normal weight $\omega(F_x)$.  Summing over all normal vertices gives the second formula, and grouping the vertex-hit terms into antipodal pairs gives the first formula.
\end{proof}

\begin{corollary}[Normal-fan invariance]
The defect $\delta(P)$ is determined by the antipodal spherical normal cell decomposition.  More precisely, it is determined by the embedded normal graph $G_P\subset S^2$ together with the antipodal map and the cell of $\mathcal N(P)$ containing $-x$ for each normal vertex $x$.
\end{corollary}

\begin{corollary}[Near-maximum gap]
Let $M(P)=2E(P)-f_2(P)/2$.  If $f_2(P)$ is even and $P$ is not facet-paired, then
\[
        \delta(P)\le M(P)-3.
\]
If $f_2(P)$ is odd, then
\[
        \delta(P)\le \lfloor M(P)\rfloor-1.
\]
\end{corollary}

\begin{proof}
Write the co-deficiency as
\[
        2(M(P)-\delta(P))
        =
        \sum_{F\in\mathcal F_0(P)}(2e(F)-1)
        +
        \sum_{F\in\mathcal F_1(P)}(2e(F)-3).
\]
Each summand on the right is an odd positive integer, and a facet in $\mathcal F_1(P)$ contributes at least $3$, while a facet in $\mathcal F_0(P)$ contributes at least $5$.
If $f_2(P)$ is even and $P$ is not facet-paired, then the number of facets outside $\mathcal F_2(P)$ is a positive even integer, because facets in $\mathcal F_2(P)$ occur in pairs.  Hence the co-deficiency is at least $6$, which gives $\delta(P)\le M(P)-3$.
If $f_2(P)$ is odd, at least one facet lies outside $\mathcal F_2(P)$, so $M(P)-\delta(P)\ge3/2$.  Since $M(P)$ is then a half-integer and $\delta(P)$ is an integer, this is equivalent to $\delta(P)\le\lfloor M(P)\rfloor-1$.
\end{proof}

\begin{corollary}[Small-defect structure]
Let $m(P)$ be the number of unordered opposite facet pairs of $P$, and let $f_1^*=|\mathcal F_1(P)|$.  Then
\[
        \delta(P)\ge f_1^*+5m(P).
\]
In particular, if $0<\delta(P)<5$, then $m(P)=0$ and $\delta(P)=f_1^*$.  If $\delta(P)=5$, then either $m(P)=0$ and $f_1^*=5$, or $m(P)=1$, the unique opposite facet pair consists of two triangular facets, and $f_1^*=0$.
\end{corollary}

\begin{proof}
Every facet--facet opposite pair $\{F,G\}$ contributes $e(F)+e(G)-1\ge3+3-1=5$, and every edge-opposite facet contributes $1$.  This proves the inequality.  The two stated consequences follow immediately by comparing with the threshold $5$.
\end{proof}

\subsection{Local square-belt modules}

The local weight in Theorem~\ref{thm:exact-local-defect} is also the rank of a natural local square-belt module.

Let \(\{F,G\}\in\mathcal O_{\rm ex}(P)\).  Let \(K_{V(F),V(G)}\) be the complete bipartite graph with parts \(V(F)\) and \(V(G)\), and set
\[
        C_2(F,G)=
        \left\langle S_{e,f}:e\in E(F),\ f\in E(G)\right\rangle_{\mathbb Z}.
\]
Here \(S_{e,f}\) denotes the square cell associated with the edge pair \((e,f)\).  Define
\[
        K(F,G)=
        \ker\left(\partial:C_2(F,G)\to C_1(K_{V(F),V(G)};\mathbb Z)\right).
\]

\begin{proposition}[Local square-belt rank]\label{prop:local-square-belt-rank}
For every exact opposite face pair \(\{F,G\}\), the group \(K(F,G)\) is free abelian of rank
\[
        \rank K(F,G)=
        e(F)e(G)-(v(F)-1)(v(G)-1).
\]
\end{proposition}

\begin{proof}
For \(e=[v_0,v_1]\subset F\) and \(f=[w_0,w_1]\subset G\), the boundary of \(S_{e,f}\) is the elementary four-cycle
\[
        (v_0,w_0)-(v_1,w_0)+(v_1,w_1)-(v_0,w_1)
\]
inside \(K_{V(F),V(G)}\).  Identifying the cycle space of this complete bipartite graph with the group of integer matrices on \(V(F)\times V(G)\) with row sums and column sums zero, this boundary is
\[
        (\mathbf e_{v_1}-\mathbf e_{v_0})\otimes
        (\mathbf e_{w_1}-\mathbf e_{w_0}).
\]
The incidence subgroup generated by the differences \(\mathbf e_{v_1}-\mathbf e_{v_0}\) for edges of the connected graph \(F^{(1)}\) has rank \(v(F)-1\).  Similarly the incidence subgroup generated by the edges of \(G^{(1)}\) has rank \(v(G)-1\).  Therefore the image of the local boundary map has rank
\[
        (v(F)-1)(v(G)-1).
\]
The domain has rank \(e(F)e(G)\), so the kernel has the stated rank.  Since it is a subgroup of a free abelian group, it is free.
\end{proof}

Combining Theorem~\ref{thm:exact-local-defect} and Proposition~\ref{prop:local-square-belt-rank}, one may write
\[
        \delta(P)=
        \sum_{\{F,G\}\in\mathcal O_{\rm ex}(P)}
        \rank K(F,G).
\]
This rank identity should not be confused with a canonical decomposition of \(H_2(X(P);\mathbb Z)\) as a direct sum of the local modules \(K(F,G)\).  Such decompositions require additional choices or filtrations; the prism family gives one controlled instance of this phenomenon.

\section{A generic Euler relation in all dimensions}

Let $\mathcal N(P)$ be the spherical normal fan of $P$, and let $-\mathcal N(P)$ be its antipodal reflection.  Their common refinement
\[
        \mathcal C(P)=\mathcal N(P)\wedge(-\mathcal N(P))
\]
is a cell decomposition of $S^{d-1}$.  Each relatively open cell of $\mathcal C(P)$ has a unique exact label $(F,G)$, where
\[
        F=F_P(u),
        \qquad
        G=F_P(-u)
\]
for $u$ in that cell.

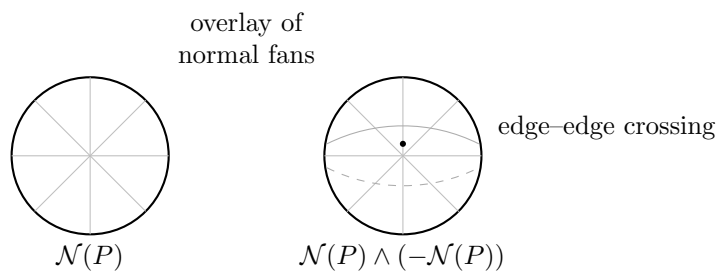
\begin{figure}[ht]
\centering
\begin{tikzpicture}[scale=0.88,every node/.style={font=\small}]
  % left fan
  \draw[thick] (0,0) circle (1.18);
  \foreach \a in {0,45,90,135} {
    \draw[gray!50] (0,0)--({1.18*cos(\a)},{1.18*sin(\a)});
    \draw[gray!50] (0,0)--({-1.18*cos(\a)},{-1.18*sin(\a)});
  }
  \node at (0,-1.52) {$\mathcal N(P)$};

  % right overlay
  \begin{scope}[shift={(4.7,0)}]
    \draw[thick] (0,0) circle (1.18);
    \draw[gray!50] (0,-1.18)--(0,1.18);
    \draw[gray!50] (-1.18,0)--(1.18,0);
    \draw[gray!50] (-0.84,-0.84)--(0.84,0.84);
    \draw[gray!50] (-0.84,0.84)--(0.84,-0.84);
    \draw[gray!65] (-1.14,0.18).. controls (-0.42,0.54) and (0.42,0.54)..(1.14,0.18);
    \draw[gray!65,dashed] (-1.14,-0.18).. controls (-0.42,-0.54) and (0.42,-0.54)..(1.14,-0.18);
    \fill (0,0.18) circle (1.3pt);
    \node[anchor=west] at (1.28,0.42) {edge--edge crossing};
    \node at (0,-1.52) {$\mathcal N(P)\wedge(-\mathcal N(P))$};
  \end{scope}

  % title above, no arrow to avoid collisions
  \node[align=center] at (2.35,1.78) {overlay of\\normal fans};
\end{tikzpicture}
\caption{The generic Euler count comes from the common refinement of the spherical normal fan and its antipodal reflection. In dimension \(3\), antipodal edge-midpoint pairs appear as crossings in this overlay.}
\label{fig:normal-fan-overlay}
\end{figure}

\begin{definition}[Transverse antipodal position]
We say that $P$ is in transverse antipodal position if, for every exact ordered antipodal pair $(F,G)$, the intersection
\[
        N^\circ(F)\cap\bigl(-N^\circ(G)\bigr)
\]
is a relatively open cell of dimension
\[
        d-1-\dim F-\dim G.
\]
In particular, exact intersections occur only when $\dim F+\dim G\le d-1$.
\end{definition}

\begin{theorem}[Generic antipodal Euler relation]\label{thm:generic-d}
Let $P\subset\mathbb R^d$ be a convex $d$-polytope in transverse antipodal position.  Then
\[
\boxed{
\sum_{(F,G)\ \mathrm{exact}}
(-1)^{d-1-\dim F-\dim G}
=
1+(-1)^{d-1},
}
\]
where the sum is over exact ordered antipodal pairs of proper faces of $P$.
\end{theorem}

\begin{proof}
Every relatively open cell of $\mathcal C(P)$ has a unique label $(F,G)$ satisfying
\[
        F=F_P(u),\qquad G=F_P(-u)
\]
for every $u$ in that cell.  Conversely, an exact ordered antipodal pair is precisely a nonempty relatively open cell of this common refinement.  By the transversality assumption, the dimension of that cell is
\[
        d-1-\dim F-\dim G.
\]
Thus the left-hand side is exactly the alternating sum of the cells of $\mathcal C(P)$.  This is the Euler characteristic of $S^{d-1}$:
\[
        \chi(S^{d-1})=1+(-1)^{d-1}.
\]
\end{proof}

\subsection{Specialization to the IMO C7 identity}

Let $P\subset\mathbb R^3$ be a C7-generic convex polyhedron. Here C7-generic means that no two edges of $P$ are parallel and no edge of $P$ is parallel to a non-adjacent face, as in the original shortlist problem. Write $V,E,F$ for the numbers of vertices, edges, and faces of $P$.  Let $A$ be the number of unordered antipodal pairs of vertices and $B$ the number of unordered antipodal pairs of edge midpoints.

The common refinement of the outward and inward normal fans is a cell decomposition of $S^2$.  In the C7-generic situation its cells can be counted as follows.

\begin{itemize}[leftmargin=2em]
\item Its two-dimensional cells correspond to ordered antipodal vertex pairs.  Hence there are $2A$ regions.
\item Its zero-dimensional cells consist of the two normal points associated with each face, together with the two oriented crossings associated with each antipodal edge-midpoint pair.  Hence there are $2F+2B$ vertices.
\item Its one-dimensional cells are obtained by subdividing the original $2E$ normal arcs.  Each antipodal edge-midpoint pair gives two crossings, and each crossing splits two arcs; hence there are $2E+4B$ one-cells.
\end{itemize}

Euler's formula on $S^2$ gives
\[
        (2F+2B)-(2E+4B)+2A=2.
\]
Therefore
\[
        A-B=E-F+1.
\]
Using Euler's formula for $P$,
\[
        V-E+F=2,
\]
we obtain
\[
\boxed{A-B=V-1.}
\]
This is precisely the identity in IMO 2006 Shortlist C7~\cite{IMO2006SL}.

\section{The antipodal square complex in dimension three}\label{sec:square-complex}

The generic identity suggests that $V-1-A+B$ measures the failure of the C7 equality in degenerate configurations.  We now encode this quantity topologically.

\begin{definition}[Antipodal graph and square complex]
Let $P\subset\mathbb R^3$ be a convex polyhedron.

The \emph{antipodal graph} $G_A(P)$ has vertex set $\Vertices(P)$ and an edge $vw$ whenever $v,w$ are distinct antipodal vertices.

The \emph{antipodal square complex} $X(P)$ is the two-dimensional finite CW complex obtained from $G_A(P)$ by attaching one square for every antipodal pair of edge midpoints.  If $e=vv'$ and $f=ww'$ are such a pair of edges, we attach a square along the cycle
\[
        v-w-v'-w'-v.
\]

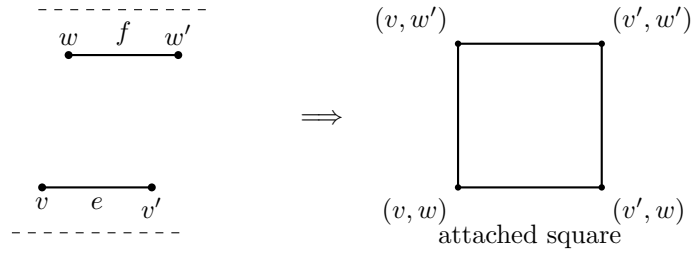
\begin{figure}[ht]
\centering
\begin{tikzpicture}[scale=1.0, every node/.style={font=\small}]
  \coordinate (v) at (0,0);
  \coordinate (vp) at (1.45,0);
  \coordinate (w) at (0.35,1.75);
  \coordinate (wp) at (1.8,1.75);

  \draw[thick] (v)--(vp);
  \draw[thick] (w)--(wp);
  \fill (v) circle (1.5pt);
  \fill (vp) circle (1.5pt);
  \fill (w) circle (1.5pt);
  \fill (wp) circle (1.5pt);

  \node[below] at (v) {$v$};
  \node[below] at (vp) {$v'$};
  \node[above] at (w) {$w$};
  \node[above] at (wp) {$w'$};
  \node[below] at ($0.5*(v)+0.5*(vp)$) {$e$};
  \node[above] at ($0.5*(w)+0.5*(wp)$) {$f$};

  \draw[dashed] (-0.4,-0.60)--(1.9,-0.60);
  \draw[dashed] (-0.05,2.35)--(2.20,2.35);
  \node at (3.7,0.9) {$\Longrightarrow$};

  \coordinate (A) at (5.5,0);
  \coordinate (B) at (7.4,0);
  \coordinate (C) at (7.4,1.9);
  \coordinate (D) at (5.5,1.9);
  \draw[thick] (A)--(B)--(C)--(D)--cycle;
  \fill (A) circle (1.2pt); \fill (B) circle (1.2pt); \fill (C) circle (1.2pt); \fill (D) circle (1.2pt);
  \node[below left] at (A) {$(v,w)$};
  \node[below right] at (B) {$(v',w)$};
  \node[above right] at (C) {$(v',w')$};
  \node[above left] at (D) {$(v,w')$};
  \node at (6.45,-0.65) {attached square};
\end{tikzpicture}
\caption{An antipodal pair of edge midpoints produces one square cell in $X(P)$.}
\label{fig:square-cell}
\end{figure}

If different edge pairs yield the same boundary cycle, they are still regarded as distinct two-cells. If the midpoints of two edges are antipodal, then the two edges are disjoint: a supporting plane through the midpoint of an edge contains the whole edge, and the two opposite supporting planes of a full-dimensional polyhedron are distinct parallel planes. Hence the attaching cycle has four distinct vertices.
\end{definition}

\begin{lemma}[Euler characteristic]\label{lem:euler-square}
For every convex polyhedron $P$,
\[
        \chi(X(P))=V(P)-A(P)+B(P).
\]
\end{lemma}

\begin{proof}
The complex has $V(P)$ zero-cells, $A(P)$ one-cells, and $B(P)$ two-cells by construction.
\end{proof}

\begin{lemma}[Connectedness of the antipodal graph]\label{lem:connected}
For every convex polyhedron $P$, the antipodal graph $G_A(P)$ is connected.
\end{lemma}

\begin{proof}
Let $uv$ be an edge of the ordinary one-skeleton of $P$.  Choose a supporting plane whose intersection with $P$ contains $uv$.  A parallel supporting plane on the opposite side of $P$ contains at least one vertex $w$.  Then both pairs $u,w$ and $v,w$ are antipodal.  Hence the edge $uv$ is replaced in $G_A(P)$ by a path
\[
        u-w-v.
\]
Since the one-skeleton of a convex polyhedron is connected, so is $G_A(P)$.
\end{proof}

\begin{proposition}[Homological form of the defect]\label{prop:defect-homology}
For every convex polyhedron $P\subset\mathbb R^3$,
\[
\boxed{
        \delta(P)=\beta_2(X(P);\mathbb Q)-\beta_1(X(P);\mathbb Q).
}
\]
\end{proposition}

\begin{proof}
By Lemma~\ref{lem:connected}, $X(P)$ is connected, so $\beta_0(X(P);\mathbb Q)=1$.  Hence
\[
        \chi(X(P))
        =1-\beta_1(X(P);\mathbb Q)+\beta_2(X(P);\mathbb Q).
\]
On the other hand, by Lemma~\ref{lem:euler-square},
\[
        \chi(X(P))=V(P)-A(P)+B(P).
\]
Subtracting $1$ gives
\[
        V(P)-1-A(P)+B(P)
        =\beta_2(X(P);\mathbb Q)-\beta_1(X(P);\mathbb Q).
\]
\end{proof}

\begin{remark}[From the weak target to the strong theorem]
Proposition~\ref{prop:defect-homology} shows that the nonnegativity of the defect is equivalent to the inequality
\[
        \beta_2(X(P);\mathbb Q)\ge \beta_1(X(P);\mathbb Q).
\]
A priori this is weaker than asking for
\[
        H_1(X(P);\mathbb Q)=0.
\]
The support blow-up argument in Section~\ref{sec:support-blowup} proves a stronger field-valued directed vanishing theorem, and then the integral homology of $X(P)$ is computed from the resulting two-sheeted cover.  Consequently, for every convex polyhedron,
\[
        \delta(P)=\beta_2(X(P);\mathbb Q).
\]
\end{remark}

\section{Examples}

This section records only hand-verifiable examples.  The examples are not used in the proof of the main theorem; they are included to show how the invariant behaves in basic cases.

\begin{example}[Tetrahedron]\label{ex:tetra}
Let $P$ be a tetrahedron.  Every pair of distinct vertices is antipodal: given two vertices $v,w$, one may choose a linear functional which is maximized at $v$ and minimized at $w$.  Hence
\[
        A=\binom{4}{2}=6.
\]
The antipodal pairs of edge midpoints are exactly the three pairs of opposite edges.  Thus
\[
        B=3.
\]
Since $V=4$,
\[
        \delta(P)=V-1-A+B=4-1-6+3=0.
\]
Thus a tetrahedron has zero antipodal defect.
\end{example}

\begin{example}[The cube]\label{ex:cube}
Let $P=[-1,1]^3$.  We claim that
\[
        V=8,
        \qquad
        A=28,
        \qquad
        B=42,
        \qquad
        \delta(P)=21.
\]
First, every pair of distinct vertices of the cube is antipodal.  Indeed, if $p,q\in\{\pm1\}^3$ are distinct, choose a vector $u$ whose $i$-th coordinate is positive when $p_i=1,q_i=-1$, negative when $p_i=-1,q_i=1$, and zero when $p_i=q_i$.  Then $p\in F_P(u)$ and $q\in F_P(-u)$.  Hence
\[
        A=\binom{8}{2}=28.
\]

It remains to count antipodal pairs of edge midpoints.  The cube has four edges in each of the three coordinate directions.  Consider first two edges in the same direction, say the $x$-direction.  Such an edge is determined by fixed signs $(y,z)=(a,b)$ with $a,b\in\{\pm1\}$.  Its normal cone on the sphere is generated by the two coordinate normals $a e_y$ and $b e_z$.  Two distinct $x$-parallel edges have sign pairs $(a,b)$ and $(c,d)$; since the sign pairs are distinct, either $a=-c$ or $b=-d$, and the two edge normal cones have opposite directions in common.  Thus all six pairs among the four $x$-edges are antipodal.  The same holds in the $y$- and $z$-directions, contributing
\[
        3\binom{4}{2}=18
\]
pairs.

Now consider edges in different directions, say an $x$-edge and a $y$-edge.  An $x$-edge has fixed signs $(y,z)=(a,b)$, while a $y$-edge has fixed signs $(x,z)=(c,d)$.  Their normal cones can meet oppositely only in the $z$-direction, and this happens exactly when $b=-d$.  Hence exactly half of the $4\cdot4=16$ such pairs are antipodal, giving $8$ pairs for the direction pair $(x,y)$.  There are three unordered pairs of coordinate directions, contributing
\[
        3\cdot 8=24
\]
additional pairs.  Therefore
\[
        B=18+24=42.
\]
Consequently
\[
        \delta(P)=8-1-28+42=21.
\]
The cube therefore shows that the C7 equality need not persist outside generic position.
\end{example}

\begin{example}[The regular octahedron]\label{ex:octahedron}
Let \(P\) be the regular octahedron.  Its four pairs of opposite triangular facets are exact facet--facet pairs.  Each pair contributes
\[
        3+3-1=5
\]
by Theorem~\ref{thm:exact-local-defect}.  Thus
\[
        \delta(P)=4\cdot5=20.
\]
Equivalently, the octahedron has \(V=6\), every pair of distinct vertices is antipodal, so \(A=\binom62=15\), and hence \(B=30\).
\end{example}

\begin{remark}[Why examples are kept modest]
Further examples can be investigated by normal-cone feasibility and cellular homology.  Since the present note is written as a proof rather than a computational report, we include only examples whose counts are transparent by hand.
\end{remark}

\section{Difference bodies and the representation viewpoint}\label{sec:difference-body}

The second official solution of C7 uses the difference body
\[
        D=P-P=\{x-y:x,y\in P\}.
\]
This viewpoint is not merely another proof of the generic case.  It gives a useful way to locate degeneracy: outside generic position, a boundary point of $D$ may have a positive-dimensional family of representations as a difference of two antipodal points of $P$.  Families of fibers of linear projections of polytopes are studied systematically in the theory of fiber polytopes \cite{BilleraSturmfels1992}.  Here we use only the geometry of the individual fibers of the difference map, not the fiber-polytope construction itself.

We keep the notation
\[
        h_P(u)=\max_{x\in P}\langle u,x\rangle,
        \qquad
        F_P(u)=\{x\in P:\langle u,x\rangle=h_P(u)\}.
\]

\begin{proposition}[Antipodality and the boundary of the difference body]\label{prop:boundary-difference}
Let $P\subset\mathbb R^3$ be a convex polyhedron and let $x,y\in P$.  Then $x$ and $y$ are antipodal if and only if
\[
        x-y\in \partial(P-P).
\]
\end{proposition}

\begin{proof}
Let $D=P-P$.  Suppose first that $x$ and $y$ are antipodal.  Then there exists a nonzero vector $u$ such that
\[
        x\in F_P(u),\qquad y\in F_P(-u).
\]
Hence
\[
        \langle u,x-y\rangle
        =h_P(u)+h_P(-u)
        =h_D(u),
\]
so $x-y\in F_D(u)\subset\partial D$.

Conversely, suppose that $x-y\in\partial D$.  Then there exists $u\ne0$ such that
\[
        \langle u,x-y\rangle=h_D(u).
\]
Since $D=P+(-P)$, its support function satisfies
\[
        h_D(u)=h_P(u)+h_P(-u).
\]
On the other hand,
\[
        \langle u,x\rangle\le h_P(u),
        \qquad
        \langle -u,y\rangle\le h_P(-u).
\]
The equality for the sum can hold only if both inequalities are equalities.  Thus
\[
        x\in F_P(u),\qquad y\in F_P(-u),
\]
so $x$ and $y$ lie on two opposite supporting planes of $P$.
\end{proof}

\begin{proposition}[Support faces of the difference body]\label{prop:face-difference-body}
For every nonzero vector $u$,
\[
        F_{P-P}(u)=F_P(u)-F_P(-u).
\]
\end{proposition}

\begin{proof}
This is the standard face formula for Minkowski sums:
\[
        F_{A+B}(u)=F_A(u)+F_B(u).
\]
Since $P-P=P+(-P)$ and
\[
        F_{-P}(u)=-F_P(-u),
\]
we get
\[
        F_{P-P}(u)=F_P(u)+F_{-P}(u)=F_P(u)-F_P(-u).
\]
\end{proof}

\begin{definition}[Difference representation space]\label{def:difference-representation-space}
Define
\[
        \mathcal R(P)=\{(x,y)\in P\times P:x-y\in\partial(P-P)\}.
\]
By Proposition~\ref{prop:boundary-difference}, this is equivalently the space of ordered antipodal pairs of points of $P$.  Let
\[
        \pi:\mathcal R(P)\longrightarrow \partial(P-P),
        \qquad
        \pi(x,y)=x-y
\]
be the difference map.
\end{definition}

\begin{proposition}[Convex fibers of the difference map]\label{prop:convex-fibers}
The map
\[
        \pi:\mathcal R(P)\to\partial(P-P)
\]
is a surjective PL map.  For every $z\in\partial(P-P)$, the fiber
\[
        \pi^{-1}(z)=\{(x,y)\in P\times P:x-y=z\}
\]
is a nonempty compact convex polytope.  In particular, every fiber is contractible.
\end{proposition}

\begin{proof}
The map $(x,y)\mapsto x-y$ is linear on $P\times P$, and $\mathcal R(P)$ is the inverse image of the polyhedral set $\partial(P-P)$ under this linear map.  Thus $\pi$ is a PL map.  It is surjective because every $z\in P-P$ has at least one representation $z=x-y$, and if $z\in\partial(P-P)$ then such pairs are antipodal by Proposition~\ref{prop:boundary-difference}.

For fixed $z$, the fiber is
\[
        (P\times P)\cap\{(x,y):x-y=z\},
\]
the intersection of a compact polytope with an affine subspace.  Hence it is a compact convex polytope.  Since $z\in P-P$, it is nonempty.
\end{proof}

\begin{remark}[Convex fibers and the discrete model]\label{rem:convex-fibers-to-homotopy}
The fibers of $\pi$ have no internal topology: they are convex.  This
suggests that the boundary of the difference body is the natural continuous
parameter space behind antipodal pairs.  The proof below uses only a finite
polyhedral shadow of this picture, obtained by recording a witnessing normal
direction for each cell of the directed square complex.
\end{remark}

The support blow-up constructed in Section~\ref{sec:support-blowup} should be viewed as a discrete replacement for this representation-space picture.  Instead of mapping the directed square complex directly to $\mathcal R(P)$, we keep track of a normal direction witnessing each cell of $\Xdir(P)$.  This produces an incidence space over $S^2$ whose fibers are exactly the local directed complexes appearing in the proof of the integral theorem.

\section{The support blow-up and the integral homological theorem}\label{sec:support-blowup}

We now prove the main theorem.  The point is to replace the noncanonical PL disk patching picture by a natural incidence space over the normal sphere.  We first prove a directed vanishing theorem over every field, then convert it into the integral homology calculation for the undirected square complex.

Throughout this section $P\subset\mathbb R^3$ is a full-dimensional compact convex polyhedron.  For a nonempty face $F\le P$, write
\[
        N^+(F)=\{u\in S^2:F\subseteq F_P(u)\},
        \qquad
        N^-(F)=-N^+(F).
\]
These are closed spherical normal cells.  They are spherically convex, hence contractible whenever nonempty.

\subsection{The directed square complex}

Let $\Xdir(P)$ be the directed antipodal square complex.  Its vertices are $v^+$ and $v^-$, one pair for each vertex $v$ of $P$.  It has an edge
\[
        [v^+,w^-]
\]
for each ordered antipodal pair of distinct vertices $(v,w)$, equivalently whenever
\[
        N^+(v)\cap N^-(w)\ne\varnothing.
\]
For each ordered pair of edges $(e,f)$, where $e=vv'$ and $f=ww'$, whose midpoints are antipodal, attach a square along
\[
        v^+-w^- - v'^+ - w'^- - v^+.
\]
The involution $v^+\leftrightarrow v^-$ extends freely to $\Xdir(P)$, and the quotient is the undirected antipodal square complex $X(P)$.

\subsection{Local directed complexes}

Let $u\in S^2$, and put
\[
        F=F_P(u),\qquad G=F_P(-u).
\]
Define the local directed complex $\Xdir(F,G)$ to be the subcomplex of $\Xdir(P)$ spanned by:
\begin{itemize}[leftmargin=2em]
\item the vertices $v^+$, with $v\in\Vertices(F)$, and $w^-$, with $w\in\Vertices(G)$;
\item all edges $[v^+,w^-]$ between these two sets of vertices;
\item all squares corresponding to pairs of edges $e\subseteq F$, $f\subseteq G$.
\end{itemize}
These edges and squares are indeed cells of $\Xdir(P)$, because the same direction $u$ witnesses the corresponding antipodality.

\begin{lemma}[Local directed filling]\label{lem:local-directed-filling}
For every $u\in S^2$, with $F=F_P(u)$ and $G=F_P(-u)$, the complex $\Xdir(F,G)$ is connected and satisfies
\[
        H_1(\Xdir(F,G);\mathbb Z)=0.
\]
Consequently, $H_1(\Xdir(F,G);k)=0$ for every field $k$.
\end{lemma}

\begin{proof}
The one-skeleton of $\Xdir(F,G)$ is the complete bipartite graph
\[
        K(\Vertices(F),\Vertices(G)).
\]
It is connected because both support faces are nonempty.

If one of $F,G$ is a vertex, this graph is a star or a single edge, and hence has trivial first homology.  Assume now that both $F$ and $G$ have at least one edge.

The cycle space of a complete bipartite graph is generated over $\mathbb Z$ by four-cycles
\[
        a^+-x^- - b^+ - y^- - a^+,
\]
where $a,b\in\Vertices(F)$ and $x,y\in\Vertices(G)$.  Choose edge paths
\[
        a=a_0,a_1,\ldots,a_r=b
\]
in the one-skeleton of $F$, and
\[
        x=x_0,x_1,\ldots,x_s=y
\]
in the one-skeleton of $G$.  With coherent orientations, the four-cycle above is equal in cellular chains over $\mathbb Z$ to a sum of elementary square boundaries of the form
\[
        a_i^+ - x_j^- - a_{i+1}^+ - x_{j+1}^- - a_i^+,
        \qquad
        0\le i<r,
        \quad
        0\le j<s.
\]
Each such square exists because $a_i a_{i+1}$ is an edge of $F$ and $x_j x_{j+1}$ is an edge of $G$.  Therefore every generator of the cycle space is an integral boundary.  Hence
\[
        H_1(\Xdir(F,G);\mathbb Z)=0.
\]
The statement over any field follows by the universal coefficient theorem.
\end{proof}

\subsection{The support blow-up}

For each open cell $c$ of $\Xdir(P)$, write $|c|$ for the corresponding closed cell, and define its support set $S(c)\subseteq S^2$ as follows:
\[
        S(v^+)=N^+(v),
        \qquad
        S(w^-)=N^-(w),
\]
\[
        S([v^+,w^-])=N^+(v)\cap N^-(w),
\]
and, if $q(e,f)$ is the square associated with an ordered pair of edges $(e,f)$, set
\[
        S(q(e,f))=N^+(e)\cap N^-(f).
\]
By construction, every $S(c)$ is nonempty. It is a closed spherical convex polyhedron, hence contractible. Moreover, if $c$ is a face of $d$, then
\[
        S(d)\subseteq S(c).
\]

Define the support blow-up
\[
        \mathcal B(P)=
        \bigcup_{c\in\operatorname{Cell}(\Xdir(P))}
        |c|\times S(c)
        \subseteq \Xdir(P)\times S^2.
\]

\begin{lemma}[Polyhedral structure of the support blow-up]\label{lem:blowup-polyhedral}
The space $\mathcal B(P)$ is a compact polyhedron.  Moreover, the coordinate projections
\[
        p:\mathcal B(P)\to \Xdir(P),
        \qquad p(\xi,u)=\xi,
\]
and
\[
        q:\mathcal B(P)\to S^2,
        \qquad q(\xi,u)=u,
\]
are proper surjections.
\end{lemma}

\begin{proof}
The finite family of closed spherical polyhedra $S(c)$ is contained in the common refinement of the outward normal fan and the antipodally reflected normal fan.  Take a barycentric subdivision of that finite spherical polyhedral complex.  This gives a triangulation $T_S$ of $S^2$ in which every $S(c)$ is a subcomplex.  Likewise, after barycentrically subdividing the finite square complex $\Xdir(P)$, obtain a triangulation $T_X$ in which every closed cell $|c|$ is a subcomplex.

Use a standard product triangulation of $T_X\times T_S$, chosen compatibly on all products of faces.  For each cell $c$, the product $|c|\times S(c)$ is then a subpolyhedron, and the union defining $\mathcal B(P)$ is a finite union of subcomplexes of one common finite triangulation.  Hence $\mathcal B(P)$ is a compact polyhedron.

The coordinate projections are continuous.  They are proper because the domain is compact and the targets are Hausdorff.  The map $p$ is surjective because, for the unique open cell $c$ containing a given point $\xi$, the support set $S(c)$ is nonempty.  The map $q$ is surjective because for every $u\in S^2$ the support faces $F_P(u)$ and $F_P(-u)$ contain vertices, and the corresponding directed vertex pair belongs to the fiber over $u$.
\end{proof}

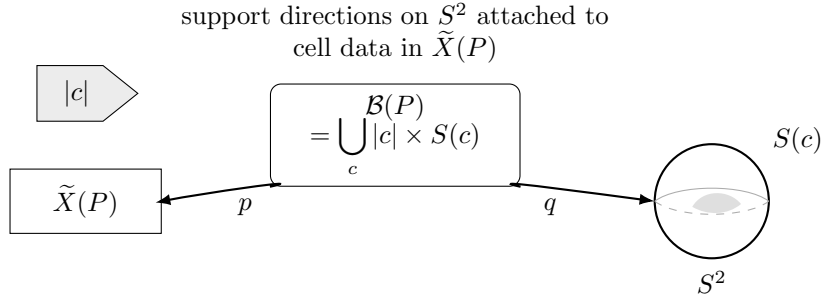
\begin{figure}[ht]
\centering
\begin{tikzpicture}[scale=0.92,every node/.style={font=\small}]
  % left directed complex block
  \node[draw,minimum width=2.0cm,minimum height=0.85cm,align=center] (X) at (-4.45,0.0) {$\Xdir(P)$};
  \draw[fill=gray!15] (-5.15,1.15)--(-4.20,1.15)--(-3.70,1.55)--(-4.20,1.95)--(-5.15,1.95)--cycle;
  \node at (-4.58,1.55) {$|c|$};

  % center blow-up box
  \node[draw,rounded corners,minimum width=3.3cm,minimum height=1.20cm,align=center] (B) at (0,0.95)
    {$\mathcal B(P)$\\[-1mm] $\displaystyle=\bigcup_c |c|\times S(c)$};

  % right sphere
  \draw[thick] (4.55,0.0) circle (0.82);
  \draw[gray!65] (3.73,0.0).. controls (4.06,0.25) and (5.04,0.25)..(5.37,0.0);
  \draw[gray!65,dashed] (3.73,0.0).. controls (4.06,-0.25) and (5.04,-0.25)..(5.37,0.0);
  \fill[gray!25] (4.27,-0.08).. controls (4.42,0.18) and (4.80,0.17)..(4.97,-0.04).. controls (4.78,-0.20) and (4.49,-0.21)..cycle;
  \node at (4.55,-1.16) {$S^2$};
  \node[anchor=west] at (5.28,0.88) {$S(c)$};

  % arrows
  \draw[-{Latex[length=2mm]},thick] (B.south west).. controls (-1.15,0.35) and (-2.85,0.10)..(-3.45,-0.02) node[midway,below] {$p$};
  \draw[-{Latex[length=2mm]},thick] (B.south east).. controls (1.15,0.35) and (2.95,0.10)..(3.73,0.00) node[midway,below] {$q$};

  % explanatory text above the diagram
  \node[align=center] at (0,2.42) {support directions on \(S^2\) attached to\\cell data in \(\Xdir(P)\)};
\end{tikzpicture}
\caption{The support blow-up associates to each cell \(c\) of the directed square complex the set \(S(c)\) of normal directions that realize that cell. The projection \(p\) has contractible fibers, while the fibers of \(q\) are local directed complexes.}
\label{fig:support-blowup}
\end{figure}

\begin{lemma}[Fibers of $p$]\label{lem:p-fibers}
For every $\xi\in\Xdir(P)$, the fiber $p^{-1}(\xi)$ is contractible.
\end{lemma}

\begin{proof}
Let $c_0$ be the unique open cell of $\Xdir(P)$ whose relative interior contains $\xi$.  Then
\[
        p^{-1}(\xi)=\bigcup_{\xi\in |d|} S(d).
\]
The union includes $S(c_0)$.  If $\xi\in |d|$, then $c_0$ is a face of $d$, so
\[
        S(d)\subseteq S(c_0).
\]
Therefore
\[
        p^{-1}(\xi)=S(c_0).
\]
This set is closed spherical convex and nonempty, hence contractible.
\end{proof}

\begin{lemma}[Fibers of $q$]\label{lem:q-fibers}
For every $u\in S^2$, with
\[
        F=F_P(u),\qquad G=F_P(-u),
\]
the first projection restricts to a canonical homeomorphism
\[
        q^{-1}(u)\cong \Xdir(F,G).
\]
Consequently, for every field $k$,
\[
        \widetilde H_0(q^{-1}(u);k)=0,
        \qquad
        H_1(q^{-1}(u);k)=0.
\]
\end{lemma}

\begin{proof}
The fiber $q^{-1}(u)$ consists of all pairs $(\xi,u)$ such that $\xi\in |c|$ for some open cell $c$ with $u\in S(c)$.  The first projection identifies $q^{-1}(u)$ with
\[
        \bigcup_{u\in S(c)} |c|.
\]
This union is a subcomplex: if $c$ is a face of $d$ and $u\in S(d)$, then $S(d)\subseteq S(c)$, so $u\in S(c)$.

For vertices,
\[
        u\in S(v^+)=N^+(v)
\]
if and only if $v\in F_P(u)=F$.  Similarly,
\[
        u\in S(w^-)=N^-(w)
\]
if and only if $w\in F_P(-u)=G$.

For edges,
\[
        u\in S([v^+,w^-])
\]
if and only if $v\in F$ and $w\in G$.  Hence all bipartite edges between $\Vertices(F)^+$ and $\Vertices(G)^-$ occur.

For squares,
\[
        u\in S(q(e,f))
\]
if and only if $e\subseteq F$ and $f\subseteq G$.  Hence the squares in the fiber are exactly the squares corresponding to pairs of edges $e\subseteq F$, $f\subseteq G$.

Therefore the first projection identifies $q^{-1}(u)$ canonically with $\Xdir(F,G)$.  The stated homology properties follow from Lemma~\ref{lem:local-directed-filling}.
\end{proof}

\subsection{The topological input}

We use the following degree-one consequence of the Vietoris--Begle mapping theorem.

\begin{theorem}[Vietoris--Begle, degree-one form]\label{thm:vietoris-begle}
Let $k$ be a field, and let $f:Y\to Z$ be a proper surjective map between compact polyhedra.  Suppose that for every $z\in Z$,
\[
        \widetilde{\check H}^{0}(f^{-1}(z);k)=0,
        \qquad
        \check H^{1}(f^{-1}(z);k)=0.
\]
Then
\[
        f^*: \check H^{1}(Z;k)\longrightarrow \check H^{1}(Y;k)
\]
is an isomorphism.
\end{theorem}

\begin{remark}
This is the case $n=2$ of the Vietoris--Begle mapping theorem in \v{C}ech cohomology: if the reduced \v{C}ech cohomology of every fiber vanishes in all degrees $<n$, then $f^*$ is an isomorphism in degrees $<n$.  We only use the case $n=2$.  See, for instance, Begle~\cite{Begle1950}, Bredon~\cite{Bredon1997}, and Iversen~\cite{Iversen1986}; for a related homotopy form, see Smale~\cite{Smale1957}.

We use this theorem only for compact polyhedra.  In that category \v{C}ech, singular, simplicial, and cellular cohomology agree.  Thus the conclusion may be read as an isomorphism
\[
        f^*:H^{1}(Z;k)\longrightarrow H^{1}(Y;k)
\]
in ordinary singular cohomology.  No assertion about singular homology for arbitrary compact metric spaces is needed.
\end{remark}

\subsection{Field-valued directed vanishing}

\begin{theorem}[Directed vanishing over all fields]\label{thm:directed-field}
The directed complex $\Xdir(P)$ is connected. Moreover, for every field $k$,
\[
        H_1(\Xdir(P);k)=0.
\]
\end{theorem}

\begin{proof}
First we prove connectedness.  The map $q:\mathcal B(P)\to S^2$ is a proper map from a compact space to a Hausdorff space, hence is closed.  Its fibers are connected by Lemma~\ref{lem:q-fibers}, and the base $S^2$ is connected.  A closed surjection with connected fibers over a connected base has connected total space, so $\mathcal B(P)$ is connected.  Since $p:\mathcal B(P)\to\Xdir(P)$ is surjective, $\Xdir(P)$ is connected.

Now let $k$ be a field.  By Lemma~\ref{lem:p-fibers}, every fiber of
\[
        p:\mathcal B(P)\to\Xdir(P)
\]
is contractible.  Applying Theorem~\ref{thm:vietoris-begle} to $p$ gives an isomorphism
\[
        H^1(\Xdir(P);k)\cong H^1(\mathcal B(P);k).
\]

For every $u\in S^2$, Lemma~\ref{lem:q-fibers} identifies $q^{-1}(u)$ with the local directed complex
\[
        \Xdir(F_P(u),F_P(-u)).
\]
By Lemma~\ref{lem:local-directed-filling}, this fiber is connected and has trivial first homology over $k$.  Since $q^{-1}(u)$ is a finite subcomplex, this is equivalent to the required \v{C}ech cohomology vanishing in degrees $0$ and $1$:
\[
        \widetilde{\check H}^{0}(q^{-1}(u);k)=0,
        \qquad
        \check H^{1}(q^{-1}(u);k)=0.
\]
Applying Theorem~\ref{thm:vietoris-begle} to $q$ gives
\[
        H^1(\mathcal B(P);k)\cong H^1(S^2;k)=0.
\]
Thus
\[
        H^1(\Xdir(P);k)=0.
\]
Since $k$ is a field, the universal coefficient theorem gives
\[
        H_1(\Xdir(P);k)=0.
\]
\end{proof}

\begin{corollary}[Integral directed vanishing]\label{cor:directed-integral}
For every convex polyhedron $P\subset\mathbb R^3$,
\[
        H_1(\Xdir(P);\mathbb Z)=0.
\]
\end{corollary}

\begin{proof}
The group $H_1(\Xdir(P);\mathbb Z)$ is finitely generated because $\Xdir(P)$ is a finite CW complex.  By the universal coefficient theorem, for every field $k$,
\[
        H_1(\Xdir(P);k)
        \cong
        H_1(\Xdir(P);\mathbb Z)\otimes k,
\]
since $H_0(\Xdir(P);\mathbb Z)$ is free.  If $H_1(\Xdir(P);\mathbb Z)$ had a nonzero free part, tensoring with $\mathbb Q$ would be nonzero.  If it had nonzero $p$-torsion, tensoring with $\mathbb F_p$ would be nonzero.  Both possibilities contradict Theorem~\ref{thm:directed-field}.  Hence $H_1(\Xdir(P);\mathbb Z)=0$.
\end{proof}

\subsection{Integral homology of the undirected square complex}

\begin{proof}[Proof of Theorem~\ref{thm:main-strong}]
The involution $v^+\leftrightarrow v^-$ sends the cell associated with an ordered pair $(v,w)$, or with an ordered pair of edges $(e,f)$, to the corresponding cell associated with $(w,v)$, respectively $(f,e)$.  No cell is fixed setwise: by convention the ordered pairs are pairs of distinct objects, and geometrically a point or an edge midpoint of a full-dimensional convex polyhedron cannot lie on both members of a pair of opposite supporting planes.  Hence the action is free and cellular, and the quotient map
\[
        \Xdir(P)\to X(P)
\]
is a connected two-sheeted covering.

Let $G=\pi_1(X(P))$, and let $H=\pi_1(\Xdir(P))$ be the subgroup corresponding to this cover.  Since the cover is connected and two-sheeted, $H\triangleleft G$ and
\[
        G/H\cong\mathbb Z/2.
\]
By Corollary~\ref{cor:directed-integral},
\[
        H^{\mathrm{ab}}=H_1(\Xdir(P);\mathbb Z)=0.
\]
Thus $H=[H,H]$, and therefore
\[
        H=[H,H]\subseteq [G,G].
\]
On the other hand, $G/H\cong\mathbb Z/2$ is abelian, so every commutator in $G$ lies in $H$; hence
\[
        [G,G]\subseteq H.
\]
Consequently $H=[G,G]$, and the covering quotient gives
\[
        H_1(X(P);\mathbb Z)=G^{\mathrm{ab}}
        =G/[G,G]
        =G/H
        \cong\mathbb Z/2.
\]

By Lemma~\ref{lem:connected}, $X(P)$ is connected, so $H_0(X(P);\mathbb Z)\cong\mathbb Z$.  Since $X(P)$ is two-dimensional, $H_i(X(P);\mathbb Z)=0$ for $i\ge3$.  Finally, $H_2(X(P);\mathbb Z)$ is the kernel of the cellular boundary map
\[
        \partial_2:C_2(X(P);\mathbb Z)\to C_1(X(P);\mathbb Z),
\]
hence is a subgroup of a free abelian group and is itself free abelian.  Its rank is $\beta_2(X(P);\mathbb Q)$.  By Proposition~\ref{prop:defect-homology} and the already proved fact that $H_1(X(P);\mathbb Q)=0$, this rank is $\delta(P)$.  Therefore
\[
        H_2(X(P);\mathbb Z)\cong\mathbb Z^{\delta(P)}.
\]
This proves all parts of the theorem.
\end{proof}

\begin{proof}[Proof of Corollary~\ref{cor:intro-defect}]
By Theorem~\ref{thm:main-strong}, $H_2(X(P);\mathbb Z)\cong\mathbb Z^{\delta(P)}$.  Hence $\delta(P)\ge0$.  Equivalently,
\[
        A(P)-B(P)\le V(P)-1.
\]
The equality $\delta(P)=\beta_2(X(P);\mathbb Q)$ follows from the same theorem.
\end{proof}

\begin{proof}[Proof of Corollary~\ref{cor:zero-defect-rp2}]
By Theorem~\ref{thm:main-strong}, the only possible difference between the integral homology of $X(P)$ and that of $\mathbb{RP}^2$ occurs in degree $2$, where
\[
        H_2(X(P);\mathbb Z)\cong\mathbb Z^{\delta(P)}.
\]
Thus $X(P)$ has the integral homology of $\mathbb{RP}^2$ if and only if $\delta(P)=0$.
\end{proof}

\begin{corollary}[Cellular independence criterion for zero defect]\label{cor:independence-zero-defect}
Let $\mathcal S(P)$ be the set of antipodal edge-midpoint pairs, and orient the corresponding square cells arbitrarily.  Then
\[
        \delta(P)=0
\]
if and only if the cellular boundaries
\[
        \{\partial s:s\in\mathcal S(P)\}\subset C_1(X(P);\mathbb Q)
\]
are linearly independent over $\mathbb Q$.
\end{corollary}

\begin{proof}
Since $X(P)$ is a two-dimensional CW complex,
\[
        H_2(X(P);\mathbb Q)=\ker\bigl(\partial_2:C_2(X(P);\mathbb Q)\to C_1(X(P);\mathbb Q)\bigr).
\]
The rank of this kernel is $\delta(P)$ by Theorem~\ref{thm:main-strong}.  Hence $\delta(P)=0$ exactly when the boundaries of the square cells are linearly independent over $\mathbb Q$.
\end{proof}

\section{Integral square-boundary structure}\label{sec:integral-structure}

The integral theorem implies more than nonnegativity of the defect.  It determines the integer structure of the full square-boundary matrix.  In this section we make this precise.

\subsection{The square-boundary lattice}

Let $G_A(P)$ be the antipodal graph.  Its integral cycle lattice is
\[
        Z_1(G_A(P);\mathbb Z)
        =\ker\bigl(\partial_1:C_1(G_A(P);\mathbb Z)\to C_0(G_A(P);\mathbb Z)\bigr).
\]
Let $\mathcal S(P)$ be the set of square cells of $X(P)$, equivalently the set of antipodal edge-midpoint pairs.  Define
\[
        L_{\mathrm{sq}}(P)=\langle\partial s:s\in\mathcal S(P)\rangle_{\mathbb Z}
        \subseteq Z_1(G_A(P);\mathbb Z).
\]
Thus $L_{\mathrm{sq}}(P)$ is the lattice of integral cycles generated by the boundaries of the square cells.

Define the parity map
\[
        \ell:Z_1(G_A(P);\mathbb Z)\longrightarrow \mathbb Z/2
\]
by
\[
        \ell\!\left(\sum_e n_e e\right)=\sum_e n_e\pmod 2.
\]
The orientation of the edges is irrelevant modulo $2$.

\begin{theorem}[Parity lattice theorem]\label{thm:parity-lattice}
For every full-dimensional convex polyhedron $P\subset\mathbb R^3$,
\[
        L_{\mathrm{sq}}(P)=\ker \ell.
\]
Equivalently, the boundaries of the square cells generate precisely the integral cycles of even total edge parity in the antipodal graph.
\end{theorem}

\begin{proof}
Each square boundary contains four oriented edges, counted with signs.  Therefore its total edge coefficient is even modulo $2$, and hence
\[
        L_{\mathrm{sq}}(P)\subseteq \ker\ell.
\]

By definition of the square complex,
\[
        H_1(X(P);\mathbb Z)
        \cong
        Z_1(G_A(P);\mathbb Z)/L_{\mathrm{sq}}(P).
\]
By Theorem~\ref{thm:main-strong}, this quotient is isomorphic to $\mathbb Z/2$.

It remains to identify the quotient map with the parity map.  The directed double cover
\[
        \Xdir(P)\to X(P)
\]
is the cover determined by the cellular $1$-cocycle with values in $\mathbb Z/2$ that sends every edge of the antipodal graph to $1$.  Indeed, traversing any edge of $X(P)$ changes the sign sheet, while the boundary of every square has four edges, so this cocycle extends over all square cells.  Since the cover is connected, the induced homomorphism
\[
        H_1(X(P);\mathbb Z)\to \mathbb Z/2
\]
is nonzero.  As $H_1(X(P);\mathbb Z)\cong\mathbb Z/2$, it is an isomorphism.  Therefore the kernel of the parity map on $Z_1(G_A(P);\mathbb Z)$ is exactly $L_{\mathrm{sq}}(P)$.
\end{proof}

\begin{corollary}\label{cor:even-cycles}
Every even integral cycle in $G_A(P)$ is the boundary of an integral square-chain in $X(P)$.  Any two odd cycles in $G_A(P)$ represent the same nonzero class in $H_1(X(P);\mathbb Z)$, and for every odd cycle $\gamma$ one has
\[
        2\gamma\in L_{\mathrm{sq}}(P).
\]
\end{corollary}

\subsection{Smith normal form}

Let
\[
        \SqBd:\mathbb Z^{\mathcal S(P)}\longrightarrow C_1(G_A(P);\mathbb Z)
\]
be the cellular boundary matrix whose columns are the oriented square boundaries.  Let
\[
        r=\rank Z_1(G_A(P);\mathbb Z)=A(P)-V(P)+1,
\]
where connectedness of $G_A(P)$ was proved in Lemma~\ref{lem:connected}.

\begin{theorem}[Smith normal form theorem]\label{thm:snf}
For every full-dimensional convex polyhedron $P\subset\mathbb R^3$,
\[
        \ker \SqBd\cong\mathbb Z^{\delta(P)},
        \qquad
        \coker \SqBd\cong \mathbb Z^{V(P)-1}\oplus\mathbb Z/2.
\]
Equivalently, if
\[
        r=A(P)-V(P)+1,
\]
then the nonzero Smith invariant factors of $\SqBd$ are
\[
        \operatorname{diag}(\underbrace{1,\ldots,1}_{r-1},2).
\]
\end{theorem}

\begin{proof}
The image of $\SqBd$ is exactly $L_{\mathrm{sq}}(P)$.  By Theorem~\ref{thm:parity-lattice},
\[
        Z_1(G_A(P);\mathbb Z)/\operatorname{im}\SqBd
        \cong \mathbb Z/2.
\]
Since $G_A(P)$ is connected,
\[
        C_1(G_A(P);\mathbb Z)/Z_1(G_A(P);\mathbb Z)
        \cong \operatorname{im}\partial_1
        \cong \mathbb Z^{V(P)-1}.
\]
Thus we have a short exact sequence
\[
0\longrightarrow
Z_1(G_A(P);\mathbb Z)/\operatorname{im}\SqBd
\longrightarrow
C_1(G_A(P);\mathbb Z)/\operatorname{im}\SqBd
\longrightarrow
C_1(G_A(P);\mathbb Z)/Z_1(G_A(P);\mathbb Z)
\longrightarrow 0.
\]
The right-hand group is free abelian, so this sequence splits.  Hence
\[
        \coker \SqBd
        \cong \mathbb Z^{V(P)-1}\oplus\mathbb Z/2.
\]

The rank of $\SqBd$ is
\[
        r=\rank Z_1(G_A(P);\mathbb Z)=A(P)-V(P)+1.
\]
Since
\[
        \SqBd:\mathbb Z^{B(P)}\longrightarrow \mathbb Z^{A(P)}
\]
has cokernel with free rank
\[
        A(P)-r=V(P)-1
\]
and torsion subgroup $\mathbb Z/2$, its nonzero Smith invariant factors are $r-1$ copies of $1$ and one copy of $2$.

Finally,
\[
        \rank\ker \SqBd
        =B(P)-r
        =B(P)-(A(P)-V(P)+1)
        =\delta(P).
\]
The kernel is a subgroup of a free abelian group, hence is free.  Therefore
\[
        \ker \SqBd\cong\mathbb Z^{\delta(P)}.
\]
\end{proof}

\subsection{The directed double cover}

The directed complex also has a completely determined integral homology.

\begin{theorem}[Directed cover homology]\label{thm:directed-cover-homology}
For every full-dimensional convex polyhedron $P\subset\mathbb R^3$,
\[
H_i(\Xdir(P);\mathbb Z)\cong
\begin{cases}
\mathbb Z, & i=0,\\
0, & i=1,\\
\mathbb Z^{2\delta(P)+1}, & i=2,\\
0, & i\ge3.
\end{cases}
\]
\end{theorem}

\begin{proof}
Theorem~\ref{thm:directed-field} and Corollary~\ref{cor:directed-integral} give
\[
        H_1(\Xdir(P);\mathbb Z)=0,
\]
and $\Xdir(P)$ is connected.  Since $\Xdir(P)\to X(P)$ is a connected two-sheeted cover,
\[
        \chi(\Xdir(P))=2\chi(X(P)).
\]
By Lemma~\ref{lem:euler-square},
\[
        \chi(X(P))=V(P)-A(P)+B(P)=\delta(P)+1.
\]
Thus
\[
        \chi(\Xdir(P))=2\delta(P)+2.
\]
The complex $\Xdir(P)$ is two-dimensional, so $H_i=0$ for $i\ge3$.  Its second homology is a subgroup of the free cellular chain group $C_2(\Xdir(P);\mathbb Z)$, hence is free.  Since
\[
        \chi(\Xdir(P))=1+\rank H_2(\Xdir(P);\mathbb Z),
\]
we obtain
\[
        H_2(\Xdir(P);\mathbb Z)\cong\mathbb Z^{2\delta(P)+1}.
\]
\end{proof}

\begin{corollary}\label{cor:zero-defect-cover}
A convex polyhedron $P\subset\mathbb R^3$ has zero defect if and only if $X(P)$ has the integral homology of $\mathbb{RP}^2$.  Equivalently, $P$ has zero defect if and only if $\Xdir(P)$ has the integral homology of $S^2$.
\end{corollary}

\begin{corollary}[Rational deck decomposition]\label{cor:deck-decomposition}
Let $\tau$ be the deck involution of $\Xdir(P)\to X(P)$.  Then
\[
        H_2(\Xdir(P);\mathbb Q)
        =H_2(\Xdir(P);\mathbb Q)^+
        \oplus
        H_2(\Xdir(P);\mathbb Q)^-,
\]
with
\[
        \dim H_2(\Xdir(P);\mathbb Q)^+=\delta(P),
        \qquad
        \dim H_2(\Xdir(P);\mathbb Q)^-=\delta(P)+1.
\]
\end{corollary}

\begin{proof}
For a finite cover, transfer identifies the invariant part of rational homology with the homology of the quotient:
\[
        H_2(\Xdir(P);\mathbb Q)^+
        \cong H_2(X(P);\mathbb Q).
\]
The latter has dimension $\delta(P)$ by Theorem~\ref{thm:main-strong}.  The total dimension of $H_2(\Xdir(P);\mathbb Q)$ is $2\delta(P)+1$ by Theorem~\ref{thm:directed-cover-homology}; hence the anti-invariant dimension is $\delta(P)+1$.
\end{proof}

\subsection{The mod-two extra belt}

\begin{corollary}[Mod-two kernel dimension]\label{cor:mod2-belt}
Over $\mathbb F_2$,
\[
        \dim_{\mathbb F_2}\ker(\SqBd\otimes\mathbb F_2)=\delta(P)+1.
\]
Equivalently,
\[
        \dim_{\mathbb F_2}H_2(X(P);\mathbb F_2)=\delta(P)+1.
\]
\end{corollary}

\begin{proof}
The Smith form in Theorem~\ref{thm:snf} has one invariant factor equal to $2$.  Reducing modulo $2$ lowers the rank by one.  Since the rational nullity is $\delta(P)$, the mod-two nullity is $\delta(P)+1$.  The homological statement is the cellular interpretation of the same kernel.
\end{proof}

\section{Equality and primitive belts}\label{sec:primitive-belts}

\subsection{The normal-fan square-boundary matroid}

A rational antipodal belt is a rational square-chain with zero cellular boundary.

\begin{definition}[Rational antipodal belt]
Let $\mathcal S(P)$ be the set of square cells of $X(P)$.  A \emph{rational antipodal belt} is an element
\[
        Z=\sum_{s\in\mathcal S(P)}\lambda_s s\in C_2(X(P);\mathbb Q)
\]
such that
\[
        \partial Z=0.
\]
The vector space of rational antipodal belts is denoted
\[
        \operatorname{Belt}_{\mathbb Q}(P).
\]
\end{definition}

Since $X(P)$ is two-dimensional,
\[
        \operatorname{Belt}_{\mathbb Q}(P)=H_2(X(P);\mathbb Q).
\]
Thus Theorem~\ref{thm:main-strong} gives
\[
        \dim_{\mathbb Q}\operatorname{Belt}_{\mathbb Q}(P)=\delta(P).
\]

Orient the square cells and the edges of $G_A(P)$ arbitrarily, and let
\[
        \SqBd:\mathbb Q^{\mathcal S(P)}\to Z_1(G_A(P);\mathbb Q)
\]
be the rational square-boundary matrix.  The represented linear matroid is denoted
\[
        M_{\mathrm{belt}}(P).
\]
Changing orientations only rescales rows or columns by $-1$, so the matroid is independent of these choices.

\begin{proposition}[Matroid rank and nullity]\label{prop:matroid-rank}
The normal-fan square-boundary matroid satisfies
\[
        \rank M_{\mathrm{belt}}(P)=A(P)-V(P)+1,
\]
\[
        \operatorname{nullity}M_{\mathrm{belt}}(P)=\delta(P).
\]
In particular,
\[
        \delta(P)=0
        \quad\Longleftrightarrow\quad
        M_{\mathrm{belt}}(P)\text{ is independent}.
\]
\end{proposition}

\begin{proof}
By Theorem~\ref{thm:snf}, the rational rank of $\SqBd$ equals the rank of the cycle space of $G_A(P)$, namely
\[
        A(P)-V(P)+1.
\]
The number of columns is $B(P)$, hence the nullity is
\[
        B(P)-A(P)+V(P)-1=\delta(P).
\]
The final equivalence follows because independence of the represented matroid is exactly vanishing of the column-dependency space.
\end{proof}

Membership in the row and column sets of $\SqBd$ can be read from closed normal cells:
\[
        \{v,w\}\in E(G_A(P))
        \quad\Longleftrightarrow\quad
        N^+(v)\cap N^-(w)\ne\varnothing,
\]
\[
        \{e,f\}\in\mathcal S(P)
        \quad\Longleftrightarrow\quad
        N^+(e)\cap N^-(f)\ne\varnothing.
\]
The incidence of a row in a column is determined by vertex-edge containment.  Thus $M_{\mathrm{belt}}(P)$ is a normal-fan and face-lattice invariant.

\subsection{The circuit criterion}

\begin{definition}[Primitive rational belt]
A nonzero rational antipodal belt is \emph{primitive} if its support is inclusion-minimal among nonzero rational belts.  Equivalently, primitive rational belts are the circuits of $M_{\mathrm{belt}}(P)$.
\end{definition}

\begin{proposition}[Algorithmic circuit criterion]\label{prop:circuit-criterion}
A subset $\mathcal T\subseteq\mathcal S(P)$ supports a primitive rational belt if and only if
\[
        \rank(\SqBd|_{\mathcal T})<|\mathcal T|
\]
and for every nonempty proper subset $\mathcal T'\subsetneq \mathcal T$,
\[
        \rank(\SqBd|_{\mathcal T'})=|\mathcal T'|.
\]
\end{proposition}

\begin{proof}
This is exactly the definition of a circuit in the linear matroid represented by the columns of $\SqBd$.
\end{proof}

\subsection{Rotating calipers for polygons}

Let $Q$ be a convex $n$-gon with no two adjacent edges collinear.  Let $p(Q)$ be the number of unordered pairs of parallel edges of $Q$.

Let $a_{VV}(Q)$, $a_{VE}(Q)$, and $a_{EE}(Q)$ denote the numbers of antipodal vertex--vertex, vertex--edge, and edge--edge pairs of $Q$, respectively.  Vertex--edge here means that a vertex and an edge lie on two opposite parallel support lines; it does not mean incidence.

\begin{lemma}[Rotating-calipers counts]\label{lem:rotating-calipers}
For every convex $n$-gon $Q$,
\[
        a_{VV}(Q)=n+p(Q),
\]
\[
        a_{VE}(Q)=n+2p(Q),
\]
\[
        a_{EE}(Q)=p(Q).
\]
\end{lemma}

\begin{proof}
Consider the projective normal circle $\mathbb{RP}^1=S^1/\{\pm1\}$.  A point of $\mathbb{RP}^1$ represents a pair of parallel support lines of $Q$.  As the normal direction varies, the support faces change exactly when one of the two support lines contains an edge of $Q$.

Each edge of $Q$ gives one event direction in $\mathbb{RP}^1$.  If two edges are parallel, they give the same event direction.  In a convex polygon an unoriented edge direction occurs at most twice.  Hence the number of event directions is
\[
        n-p(Q).
\]
Among these, $p(Q)$ are edge--edge events.  Since each such event accounts for two edges, the remaining event directions are vertex--edge events, and their number is
\[
        n-2p(Q).
\]

There are $n-p(Q)$ open intervals between consecutive event directions.  On each interval, both opposite support faces are vertices, giving one vertex--vertex antipodal pair.

At an edge--edge event, suppose the two opposite support edges are
\[
        e=ab,
        \qquad
        f=xy.
\]
This event contributes one edge--edge pair, four vertex--edge pairs
\[
        a-f,
        \quad b-f,
        \quad x-e,
        \quad y-e,
\]
and two new vertex--vertex pairs among the four endpoint pairs.  The other two endpoint pairs are already counted by the two adjacent open intervals.

Therefore
\[
        a_{EE}(Q)=p(Q),
\]
\[
        a_{VE}(Q)=(n-2p(Q))+4p(Q)=n+2p(Q),
\]
\[
        a_{VV}(Q)=(n-p(Q))+2p(Q)=n+p(Q).
\]
\end{proof}

\subsection{Pyramids over polygons}

Throughout the polygonal family results, convex polygons are assumed to have no two adjacent edges collinear.

Let $Q$ be a convex $n$-gon and let
\[
        P=\operatorname{Pyr}(Q)
\]
be any pyramid over $Q$, with apex $a$ outside the base plane.  We regard $Q$ as lying in an affine base plane and write the apex in the form $a=(o,h)$ after choosing an affine coordinate $z$ transverse to that plane, with $h\ne0$.

\begin{lemma}[Antipodal correspondences in a pyramid]\label{lem:pyramid-correspondences}
Let $P=\operatorname{Pyr}(Q)$ with apex $a$.  Then the following hold, with all support-face conditions understood in the closed sense.

\begin{enumerate}[label=(\alph*),leftmargin=2em]
\item The apex $a$ is antipodal to every base vertex $v$.

\item Two base vertices $v,w$ are antipodal in $P$ if and only if they are antipodal in $Q$.

\item Two base edges $e,f$ have antipodal midpoints in $P$ if and only if they have antipodal midpoints in $Q$.

\item A lateral edge $av$ and a base edge $e$ have antipodal midpoints in $P$ if and only if $v$ and $e$ form a vertex--edge antipodal pair in $Q$, i.e.\ if and only if there exists an affine functional $\alpha$ on the base plane such that
\[
        v\in F_Q(\alpha),
        \qquad
        e\subseteq F_Q(-\alpha).
\]

\item No two lateral edges have antipodal midpoints.
\end{enumerate}
\end{lemma}

\begin{proof}
Every affine functional on the ambient three-space may be written as
\[
        L(x,z)=\alpha(x)+\lambda z,
\]
where $\alpha$ is an affine functional on the base plane.

For (a), choose an affine functional $\alpha$ on the base plane minimized at $v$.  If $h>0$, choose $\lambda$ so large that
\[
        \alpha(o)+\lambda h>\max_Q\alpha;
\]
if $h<0$, choose $\lambda$ with the opposite sign.  Then the apex is the unique upper extremum of $L$, while the minimum support face is the minimum face of $\alpha$ in the base and contains $v$.  Hence $a$ and $v$ are antipodal.

For (b), suppose first that two base vertices $v,w$ are antipodal in $P$.  Restrict a witnessing affine functional $L$ to the base plane.  Since both points lie in the base plane, this restriction witnesses that $v$ and $w$ are antipodal in $Q$.  Conversely, if $v$ and $w$ are antipodal in $Q$, choose an affine functional $\alpha$ on the base plane with
\[
        v\in F_Q(\alpha),
        \qquad
        w\in F_Q(-\alpha).
\]
Choose $\lambda$ so that
\[
        \min_Q\alpha\le \alpha(o)+\lambda h\le \max_Q\alpha.
\]
Such a value exists because $h\ne0$.  The apex then lies between the two supporting planes, and the same two base vertices are antipodal in $P$.

The proof of (c) is identical, replacing base vertices by base edges.  A witnessing functional in $P$ restricts to one in the base plane, and conversely a base-plane witness may be extended by choosing $\lambda$ so that the displayed interval condition holds.  Closed support faces cause no difficulty if the apex lies on one of the two support planes.

For (d), suppose first that the lateral edge $av$ and the base edge $e$ are antipodal in $P$.  After replacing $L$ by $-L$ if necessary, assume that $av$ lies in a maximum support face and $e$ lies in a minimum support face.  Since $L$ is constant on $av$, its restriction $\alpha$ satisfies
\[
        v\in F_Q(\alpha),
        \qquad
        e\subseteq F_Q(-\alpha).
\]
Thus $v$ and $e$ form a vertex--edge antipodal pair in $Q$.  Conversely, if such an $\alpha$ exists, choose $\lambda$ so that
\[
        L(a)=L(v)=\max_Q\alpha.
\]
Then the lateral edge $av$ is contained in a support plane of $P$, while $e$ is contained in an opposite support plane.  Hence their midpoints are antipodal.

For (e), any two lateral edges share the apex $a$.  If two such edges were contained in opposite supporting planes, then $a$ would lie on both of two distinct parallel supporting planes, which is impossible for a full-dimensional polyhedron.
\end{proof}

\begin{theorem}[Pyramid defect formula]\label{thm:pyramid-defect}
For every convex polygon $Q$,
\[
        \delta(\operatorname{Pyr}(Q))=2p(Q).
\]
Consequently,
\[
        \operatorname{Pyr}(Q)\text{ has zero defect}
        \quad\Longleftrightarrow\quad
        Q\text{ has no pair of parallel edges}.
\]
\end{theorem}

\begin{proof}
We have $V(P)=n+1$.  By Lemma~\ref{lem:pyramid-correspondences}, the antipodal vertex pairs consist of the $n$ apex--base pairs and the base--base antipodal pairs of $Q$.  Therefore
\[
        A(P)=n+a_{VV}(Q)=n+(n+p(Q))=2n+p(Q).
\]
Similarly, the antipodal edge-midpoint pairs consist of the base--base edge pairs counted by $a_{EE}(Q)$ and the lateral--base pairs counted by $a_{VE}(Q)$.  There are no lateral--lateral pairs.  Hence
\[
        B(P)=a_{VE}(Q)+a_{EE}(Q)=(n+2p(Q))+p(Q)=n+3p(Q).
\]
Thus
\[
        \delta(P)=V(P)-1-A(P)+B(P)
        =(n+1)-1-(2n+p(Q))+(n+3p(Q))=2p(Q).
\]
\end{proof}

\begin{corollary}[Regular pyramids]\label{cor:regular-pyramids}
For a regular $n$-gonal pyramid,
\[
\delta(P)=
\begin{cases}
0, & n\text{ odd},\\
n, & n\text{ even}.
\end{cases}
\]
\end{corollary}

\begin{proof}
By Theorem~\ref{thm:pyramid-defect},
\[
        \delta(P)=2p(Q),
\]
where $Q$ is the regular $n$-gon forming the base.  A regular $n$-gon has no pair of parallel edges when $n$ is odd, and has exactly $n/2$ pairs of parallel opposite edges when $n$ is even.  Hence
\[
        \delta(P)=0
        \quad\text{if }n\text{ is odd},
        \qquad
        \delta(P)=2\cdot \frac n2=n
        \quad\text{if }n\text{ is even}.
\]
\end{proof}

\begin{example}[The same face lattice with different pyramid defects]\label{ex:quadrilateral-pyramids}
Consider the two convex quadrilaterals
\[
\begin{aligned}
Q_0&=\operatorname{conv}\{(0,0),(3,0),(4,2),(0,3)\},\\
Q_1&=\operatorname{conv}\{(0,0),(3,0),(2,2),(0,2)\}.
\end{aligned}
\]
The polygon $Q_0$ has no parallel pair of sides, whereas $Q_1$ has exactly
one parallel pair, namely its two horizontal sides.  Therefore any pyramids
over $Q_0$ and $Q_1$ are combinatorially equivalent quadrilateral pyramids,
but Theorem~\ref{thm:pyramid-defect} gives
\[
        \delta(\operatorname{Pyr}(Q_0))=0,
        \qquad
        \delta(\operatorname{Pyr}(Q_1))=2.
\]
More explicitly, the two count triples are
\[
        (V,A,B)=(5,8,4)
        \qquad\text{and}\qquad
        (V,A,B)=(5,9,7).
\]
Hence the antipodal defect is not determined by the face lattice; it detects
metric information carried by the normal fan.
\end{example}

\subsection{Primitive belts of pyramids}

Fix a pair of parallel opposite base edges
\[
        e=uv,
        \qquad
        f=xy,
\]
and let $a$ be the apex.  The associated local block consists of five square cells:
\[
        B=S_{e,f},
        \qquad
        U=S_{au,f},
        \qquad
        V=S_{av,f},
        \qquad
        X=S_{ax,e},
        \qquad
        Y=S_{ay,e}.
\]

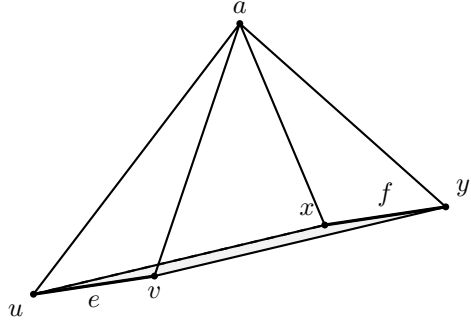
\begin{figure}[ht]
\centering
\begin{tikzpicture}[x={(1.0cm,0.15cm)},y={(0.55cm,0.30cm)},z={(0cm,1.0cm)},scale=1.0,every node/.style={font=\small}]
  \coordinate (u) at (-2.3,0,0);
  \coordinate (v) at (-0.7,0,0);
  \coordinate (y) at (2.3,1.55,0);
  \coordinate (x) at (0.7,1.55,0);
  \coordinate (a) at (0.0,0.78,3.0);

  \fill[gray!10] (u)--(v)--(y)--(x)--cycle;
  \draw[thick] (u)--(v)--(y)--(x)--cycle;
  \draw[thick] (a)--(u) (a)--(v) (a)--(x) (a)--(y);
  \draw[thick,dashed] (u)--(x);

  \draw[very thick] (u)--(v);
  \draw[very thick] (x)--(y);

  \fill (u) circle (1.3pt); \fill (v) circle (1.3pt);
  \fill (x) circle (1.3pt); \fill (y) circle (1.3pt);
  \fill (a) circle (1.3pt);

  \node[below left] at (u) {$u$};
  \node[below] at (v) {$v$};
  \node[above left] at (x) {$x$};
  \node[above right] at (y) {$y$};
  \node[above] at (a) {$a$};
  \node[below] at ($0.5*(u)+0.5*(v)$) {$e$};
  \node[above] at ($0.5*(x)+0.5*(y)$) {$f$};
\end{tikzpicture}

\vspace{0.4em}

\[
\begin{array}{c|c}
\text{square cells} &
B=S_{e,f},\quad U=S_{au,f},\quad V=S_{av,f},\quad X=S_{ax,e},\quad Y=S_{ay,e}\\[2mm]
\text{primitive circuits} &
\{B,U,V\},\quad \{B,X,Y\},\quad \{U,V,X,Y\}
\end{array}
\]
\caption{A parallel pair of base edges in a pyramid produces a local five-square block.  The geometry is separated from the algebraic list of square cells and primitive circuits.}
\label{fig:pyramid-block}
\end{figure}

With coherent orientations,
\[
        \partial B+\partial U-\partial V=0,
\]
\[
        \partial B+\partial X-\partial Y=0.
\]
Set
\[
        Z_1=B+U-V,
        \qquad
        Z_2=B+X-Y.
\]

\begin{theorem}[Pyramid belt-space decomposition]\label{thm:pyramid-belt-decomp}
For every convex polygon $Q$,
\[
        \operatorname{Belt}_{\mathbb Q}(\operatorname{Pyr}(Q))
        \cong
        \bigoplus_{\{e,f\}\parallel}\mathbb Q^2,
\]
where the direct sum is over unordered pairs of parallel base edges of $Q$.
\end{theorem}

\begin{proof}
Each parallel-edge block produces two independent belts $Z_1,Z_2$.  Distinct parallel-edge blocks do not share square cells, since a direction of a convex polygon occurs in at most one pair of opposite parallel edges.  Therefore the local belt spaces form a direct sum of dimension $2p(Q)$.

By Theorem~\ref{thm:pyramid-defect},
\[
        \delta(\operatorname{Pyr}(Q))=2p(Q).
\]
By the homological theorem,
\[
        \dim_{\mathbb Q}\operatorname{Belt}_{\mathbb Q}(\operatorname{Pyr}(Q))
        =\delta(\operatorname{Pyr}(Q)).
\]
Thus the local belts already span the full belt space, proving the decomposition.
\end{proof}

\begin{corollary}[Primitive circuits in a pyramid]\label{cor:pyramid-circuits}
In each parallel-edge block, the local belt space has exactly three primitive circuits:
\[
        \{B,U,V\},
        \qquad
        \{B,X,Y\},
        \qquad
        \{U,V,X,Y\}.
\]
Consequently, the primitive rational belts of a pyramid are exactly these local circuits over all parallel-edge blocks.
\end{corollary}

\begin{proof}
Every local belt has the form
\[
        \alpha Z_1+\beta Z_2
        =(
        \alpha+\beta)B+\alpha U-\alpha V+\beta X-\beta Y.
\]
Its support is minimal precisely in the three cases
\[
        \beta=0,
        \qquad
        \alpha=0,
        \qquad
        \alpha+\beta=0.
\]
These give the three supports listed above.  In the global direct sum, a belt with nonzero components in two different blocks cannot be primitive, because either nonzero component is itself a belt with strictly smaller support.  The global statement therefore follows from Theorem~\ref{thm:pyramid-belt-decomp}.
\end{proof}

\section{Conclusion and further questions}

The antipodal defect is therefore simultaneously a normal-fan quantity, a
second Betti number, and the nullity of an integral square-boundary map.  The
three descriptions serve different purposes: the local formula identifies
where degeneracy occurs, the square complex proves nonnegativity and retains
integral information, and the boundary lattice makes the torsion and
primitive dependencies explicit.  The pyramid formulas also show that the
defect is not determined by the face lattice alone.

Two problems appear particularly natural.

\begin{problem}[Primitive belt classification]
Give a geometric classification of the circuits of the normal-fan
square-boundary matroid \(M_{\mathrm{belt}}(P)\) for arbitrary convex
polyhedra.
\end{problem}

\begin{problem}[Higher-dimensional square complexes]
For \(d>3\), construct a finite combinatorial complex analogous to \(X(P)\)
whose homology detects a natural antipodal defect associated with the mixed
faces of \(P-P\).
\end{problem}

The polygonal suspension and wall-crossing phenomena suggested by the
examples form a separate direction and will be treated elsewhere.

\end{document}